\documentclass[10pt]{amsart}

\usepackage{bbm}
\usepackage{eucal}
\numberwithin{equation}{subsection}
\usepackage{xcolor}
\definecolor{winered}{rgb}{0.8,0,0}
\definecolor{deepblue}{rgb}{0,0,0.8}
\usepackage[colorlinks=true]{hyperref}
\hypersetup{linkcolor=black, citecolor=black}
\usepackage{amsmath}
\usepackage{amsthm}
\usepackage{amssymb}
\usepackage{mathrsfs}
\usepackage{tikz-cd}
\newtheorem{thm}{Theorem}[subsection]
\newtheorem{prop}[thm]{Proposition}
\newtheorem{cor}[thm]{Corollary}

\newtheorem{lem}[thm]{Lemma}
\theoremstyle{definition}
\newtheorem{df}[thm]{Definition}
\newtheorem{rmk}[thm]{Remark}
\newtheorem{exm}[thm]{Example}
\newtheorem{const}[thm]{Construction}
\newtheorem{quest}[thm]{Question}

\theoremstyle{remark}

\newcommand{\A}{\mathbb{A}}

\newcommand{\C}{\mathbb{C}}

\newcommand{\E}{\mathbb{E}}

\newcommand{\G}{\mathbb{G}}

\newcommand{\N}{\mathbb{N}}
\renewcommand{\P}{\mathbb{P}}
\newcommand{\Q}{\mathbb{Q}}
\newcommand{\R}{\mathbb{R}}
\newcommand{\T}{\mathbb{T}}

\newcommand{\Z}{\mathbb{Z}}

\newcommand{\cC}{\mathcal{C}}

\newcommand{\cE}{\mathcal{E}}
\newcommand{\cF}{\mathcal{F}}
\newcommand{\cG}{\mathcal{G}}

\newcommand{\cM}{\mathcal{M}}

\newcommand{\cO}{\mathcal{O}}
\newcommand{\cP}{\mathcal{P}}

\newcommand{\rD}{\mathrm{D}}

\newcommand{\rG}{\mathrm{G}}

\newcommand{\sT}{\mathscr{T}}

\newcommand{\et}{{\acute{e}t}}

\DeclareMathOperator{\Hom}{Hom}

\DeclareMathOperator{\Spec}{Spec}

\newcommand{\id}{\mathrm{id}}

\newcommand{\DM}{{\mathrm{DM}}}

\newcommand{\infSH}{\mathrm{SH}}
\newcommand{\infDM}{\mathrm{DM}}

\newcommand{\Sp}{\mathrm{Sp}}

\newcommand{\cofib}{\mathrm{cofib}}
\newcommand{\fib}{\mathrm{fib}}

\newcommand{\infCat}{\mathrm{Cat}}
\newcommand{\CAlg}{\mathrm{CAlg}}

\newcommand{\PrL}{\mathrm{Pr}^\mathrm{L}}

\newcommand{\Corr}{\mathrm{Corr}}

\newcommand{\DA}{\mathrm{DA}}

\newcommand{\divi}{\mathrm{div}}

\newcommand{\unit}{\mathbf{1}}

\newcommand{\setale}{\mathrm{s\acute{e}t}}
\newcommand{\ketale}{\mathrm{k\acute{e}t}}

\newcommand{\red}{\mathrm{red}}
\newcommand{\tor}{\mathrm{tor}}
\newcommand{\cok}{\mathrm{cok}}

\newcommand{\Sch}{\mathrm{Sch}}
\newcommand{\lSch}{\mathrm{lSch}}

\newcommand{\Sm}{\mathrm{Sm}}

\newcommand{\rank}{\mathrm{rank}}
\newcommand{\ver}{\mathrm{ver}}

\newcommand{\map}{\mathrm{map}}

\newcommand{\ol}{\overline}
\newcommand{\pt}{\mathrm{pt}}

\newcommand{\eSm}{\mathrm{eSm}}
\newcommand{\sSm}{\mathrm{sSm}}
\newcommand{\lSm}{\mathrm{lSm}}

\newcommand{\gp}{\mathrm{gp}}
\newcommand{\sat}{\mathrm{sat}}

\newcommand{\Normal}{\mathrm{N}}

\newcommand{\colimit}{\mathop{\mathrm{colim}}}

\newcommand{\Thom}{\mathrm{Th}}

\newcommand{\infSpc}{\mathrm{Spc}}

\newcommand{\boxx}{\square}

\newcommand{\colim}{\mathop{\mathrm{colim}}}

\newcommand{\ul}{\underline}
\newcommand{\Ho}{\mathrm{Ho}}

\newcommand{\ModMZ}{\mathrm{Mod}_{\mathrm{M}\mathbb{Z}}}
\newcommand{\Chow}{\mathrm{Chow}}

\newcommand{\Mod}{\mathrm{Mod}}

\newcommand{\CH}{\mathrm{CH}}
\newcommand{\KH}{\mathrm{KH}}

\newcommand{\ML}{\mathrm{M\Lambda}}
\newcommand{\MZ}{\mathrm{M}\mathbb{Z}}
\newcommand{\BM}{\mathrm{BM}}
\newcommand{\KGL}{\mathrm{KGL}}

\begin{document}
\title{Motivic six-functor formalism for log schemes}
\author{Doosung Park}
\address{Department of Mathematics and Informatics, University of Wuppertal, Germany}
\email{dpark@uni-wuppertal.de}

\begin{abstract}
We establish the motivic six-functor formalism for fs log schemes.
In particular,
we prove the base change property,
projection formula,
and Poincar\'e duality.
We also define Borel-Moore motivic homology, G-theory, and Chow homology of fs log schemes and the category of Chow motives over fs log schemes.
\end{abstract}

\maketitle

\section{Introduction}

Log geometry, developed by Deligne, Faltings, Fontaine, Illusie, and Kato, gives boundaries to schemes in an algebraic way.
This provides an effective tool to understand the limit behavior of families of schemes.
For example,
let $f\colon X\to S$ be a proper smooth morphism of schemes,
and let $S\to T$ be an open immersion of schemes.
Suppose that we want to study the limit behavior of $f$ in $T-S$.
For this,
we may consider an extension $g\colon Y\to T$ of $f$ such that $g$ is proper flat,
but imposing the smoothness condition on $g$ is often too strong.
What log geometry can help in this situation is that in good cases we can impose log structures on $Y$ and $T$ such that $g$ becomes \emph{log smooth}.
Smooth morphisms are more convenient to deal with than general morphisms,
and log smooth morphisms have a similar privilege.

\

An early application of log geometry is Tsuji's proof \cite{MR1705837} of the $C_{st}$ conjecture in $p$-adic Hodge theory.
To take advantage of log geometry not just in $p$-adic cohomology theories but in broader cohomology theories,
logarithmic motivic homotopy theory,
extending $\A^1$-homotopy theory of schemes initiated by Morel-Voevodsky \cite{MV},
is desirable.
The reason for this is that the $\infty$-category of motivic spectra is a natural place where cohomology theories can live.
Author's joint papers with Binda and {\O}stv{\ae}r \cite{logDM}, \cite{logSH} aim for this with a specialized interest in non $\A^1$-invariant cohomology theories.

\

On the other hand,
we focus on $\A^1$-homotopy theory of log schemes here.
This paper is the last chapter of the series initiated in \cite{logA1} and further continued in \cite{divspc}, \cite{logGysin}, and \cite{logshriek}.
This series covers the author's thesis \cite{ParThesis} with a more flexible setting.

\subsection{Cohomology and Borel-Moore homology of log schemes}

Before digging into the details of the paper,
let us shortly overview recent developments in the area of motivic or Chow theory of log schemes:

\begin{itemize}
\item
In  \cite[Definition 2.7]{1810.03746},
Barrott defines the log Chow group $A_*^\dagger(X)$ for fs log scheme $X$,
which can be formulated as
\[
A_*^\dagger(X)
:=
\colimit_{Y\in X_{\divi}}
\CH_*(Y).
\]
Here, the colimit runs over the category of dividing covers of $X$.
Molcho, Pandharipande, and Schmitt \cite{MR4549706} employ this definition for the moduli spaces of stable curves.
\item B{\"o}hning, Graf von Bothmer, and van Garrel \cite{MR4609665} define the prelog Chow ring for simple normal crossing scheme $X$.
They employ this to study semistable degeneration.
\item Gregory and Langer \cite{2108.02845} define log motivic cohomology of semistable varieties.
They relate this to the log Hyodo-Kato Hodge-Witt sheaves.
\item Ito, Kato, Nakayama, and Usui \cite{MR4043075} define log motives modulo homological or numerical equivalences over an fs log scheme.
Using this,
Kato, Nakayama, and Usui \cite{MR4938005} suggest a simple construction of the category of mixed motives over a field.
\item Shuklin \cite{2209.03720} constructs a motive $M(X)$ in Voevodsky's category of geometric motives $\mathrm{DM}_{gm}(k,\Q)$ for every fs log scheme $X$ over $k$.
As a consequence,
he defines log motivic cohomology of $X$ with coefficients in $\Q$.
\end{itemize}

\

In \S \ref{BM},
we extend Borel-Moore homology theories of schemes to fs log schemes based on the Grothendieck six-functor formalism for fs log schemes.
An advantage of this approach is that the Grothendieck  six-functor formalism provides a recipe for computing cohomology and Borel-Moore homology theories.

\

More precisely,
if $B$ is a base scheme and $\E$ is a commutative algebra object of $\infSH(B)$,
then we define \emph{Borel-Moore $\E$-homology  $\E^\BM(X/S)\in \Sp$ of $X$ relative to $S$} for any morphism $f\colon X\to S$ of fs log schemes.
See Definition \ref{BM.1} for the details.
On the other hand,
we can associated a cohomology theory $\E(X)\in \Sp$ for fs log schemes $X$ over $B$ according to \cite{logA1}.
We have the formal properties of $\E(X)$ as follows:
\begin{itemize}
\item ($\A^1$-invariance)
If $X$ is an fs log scheme,
then the morphism
\[
p^*
\colon
\E(X)\to \E(X\times \A^1)
\]
is an isomorphism,
where $p\colon X\times \A^1\to X$ is the projection.
\item (Invariance under verticalization)
Let $f\colon X\to S$ be an exact log smooth morphism of fs log schemes.
Then the morphism
\[
j^*\colon \E(X)\to \E(X-\partial_S X)
\]
is an isomorphism,
where $j\colon X-\partial_S X\to X$ is the open immersion removing the relative boundary $\partial_S X$.
\item (Invariance under dividing covers) If $f\colon Y\to X$ is a dividing cover of fs log schemes,
then the morphism
\[
f^*\colon \E(X)\to \E(Y)
\]
is an isomorphism.
\item (Strict Nisnevich descent)
Let
\[
Q
:=
\begin{tikzcd}
X'\ar[d,"f'"']\ar[r,"g'"]&
X\ar[d,"f"]
\\
S'\ar[r,"g"]&
S
\end{tikzcd}
\]
be a strict Nisnevich distinguished square,
i.e., $f$, $f'$, $g$, and $g'$ are strict, and $\ul{Q}$ is a Nisnevich distinguished square.
Then $\E(Q)$ is cocartesian.
\end{itemize}
The following property is an immediate consequence of the $\A^1$-invariance and invariance under verticalization:
\begin{itemize}
\item ($\square$-invariance)
If $X$ is an fs log scheme,
then the morphism
\[
p^*
\colon
\E(X)\to \E(X\times \square)
\]
is an isomorphism,
where $p\colon X\times \square\to X$ is the projection.
\end{itemize}

\

The Borel-Moore homology theory $\E^{\BM}$ behaves a bit differently.
For example,
$\E^{\BM}$ is not $\A^1$-invariant in general,
i.e.,
$\E^{\BM}(X\times \A^1) \not \simeq \E^{\BM}(X)$.
Such a thing already happens for motives with compact support, see \cite[Corollary 16.16]{MVW}.
This is a natural phenomenon since topological Borel-Moore homology is not $[0,1)$-invariant.
Likewise,
$\E^{\BM}$ is not invariant under verticalization in general.
On the other hand, $\E^{\BM}$ is $\square$-invariant and invariant under dividing covers since $\square$ and all dividing covers are proper.
Furthermore, $\E^{\BM}$ satisfies strict Nisnevich descent.

\

As special cases,
we define
motivic cohomology,
motivic Borel-Moore homology,
Chow cohomology,
Chow homology,
homotopy K-theory,
and 
G-theory of fs log schemes.
Note that Chow cohomology and Chow homology have to be distinguished even for log smooth case if the boundary is nonempty.

\

Here are the comparisons between our theories and the previous existing theories:

\begin{itemize}
\item The log Chow group in \cite{1810.03746} is not $\square$-invariant,
so it is different from both our Chow cohomology and Chow homology.
\item The prelog Chow ring in \cite{MR4609665} is different from Chow cohomology for a certain fs log scheme, see Example \ref{logChow.12}.
\item Log motivic cohomology of semistable varieties in \cite{2108.02845} is different from our motivic cohomology although they might be closely related,
see \cite[Remark 4.4.5]{logA1}.
\item
We suggest a strategy to compare the category of log motives modulo homological equivalences in \cite{MR4043075} and the category of Chow motives in Definition \ref{logChow.9} 
see Question \ref{logChow.2}.
\item Let $X$ be an fs log scheme over a field $k$.
We expect that $M(X)\in \DM(k,\Q)$ defined in \cite{2209.03720} is isomorphic to $f_!f^!\unit$,
where $f\colon X\to \Spec(k)$ is the structure morphism,
$f_!\colon \Mod_{\mathrm{M}\Q}(X)\to \Mod_{\mathrm{M}\Q}(k)$ is the exceptional direct image functor that is explained below,
and $f^!$ is its right adjoint.
\end{itemize}

\subsection{Summary of the previous paper}

Let $B$ be a finite dimensional noetherian separated base scheme throughout the paper,
and let
\[
\sT\colon (\lSch/B)^{op}
\to
\CAlg(\PrL)
\]
be a log motivic $\infty$-category in the sense of 
\cite[Definition 2.1.1]{logshriek}.
By \cite[Theorem 1.1.1]{logshriek},
\[
\infSH,
\text{ }
\DA(-,\Lambda),
\text{ }
\DA_\setale(-,\Lambda),
\text{ }
\DA_\ketale(-,\Lambda)
\]
in \cite[Definition 2.5.5]{logA1} are examples of a log motivic $\infty$-category,
where $\Lambda$ is a commutative ring.
When the exponential characteristic of a perfect field $k$ is invertible in a commutative ring $\Lambda$,
the $\infty$-category $\Mod_{\mathrm{M}\Lambda}(k)$ is equivalent to $\infDM(k,\Lambda)$ as observed in \cite[\S 1.1]{logshriek},
where $\mathrm{M}\Lambda$ denotes the motivic Eilenberg-MacLane spectrum.

\

We have the functor
\[
\sT^{ex}\colon (\lSch/B)^{op}\to \CAlg(\PrL)
\]
such that for $S\in \lSch/B$,
$\sT^{ex}(S)$ is the full subcategory generated under colimits by $M_S(X)(d)[n]$ for all $X\in \eSm/S$ and integers $d$ and $n$,
see
\cite[Definition 2.3.1]{logshriek}.
Here, $\eSm$ denotes the class of exact log smooth morphisms in $\lSch/B$.

\

In the following,
we summarize the results in \cite{logshriek} that will form the basis of the results in this paper.

\begin{thm}
\label{introduction.1}
The above functor $\sT^{ex}$ satisfies the following properties.
\begin{itemize}
\item For every morphism $f\colon X\to S$ in $\lSch/B$,
the functor $f^*:=\sT^{ex}(f)$ admits a right adjoint $f_*$.
If $f\in \eSm$,
the functor $f^*$ admits a left adjoint $f_\sharp$.
\item \textup{($\eSm$-BC)}
Let
\[
\begin{tikzcd}
X'\ar[d,"f'"']\ar[r,"g'"]&
X\ar[d,"f"]
\\
S'\ar[r,"g"]&
S
\end{tikzcd}
\]
be a cartesian square in $\lSch/B$ such that $f\in \eSm$.
Then the induced natural transformation
\[
f_\sharp'g'^*
\xrightarrow{Ex}
g^*f_\sharp
\]
is an isomorphism.
\item \textup{($\eSm$-PF)}
Let $f\colon X\to S$ be an exact log smooth morphism in $\lSch/B$.
Then the induced natural morphism
\[
f_\sharp(\cF\otimes f^*\cG)
\xrightarrow{Ex}
f_\sharp\cF\otimes \cG
\]
is an isomorphism for $\cF\in \sT^{ex}(X)$ and $\cG\in \sT^{ex}(S)$.
\item \textup{($\A^1$-inv)}
Let $p\colon X\times \A^1\to X$ be the projection,
where $X\in \lSch/B$.
Then $p^*$ is fully faithful.
\item \textup{($\divi$-inv)}
Let $f\colon X\to S$ be a dividing cover in $\lSch/B$.
Then $f^*$ is fully faithful.
\item \textup{($\ver$-inv)}
Let $f\colon X\to S$ be an exact log smooth morphism in $\lSch/B$, and let $j\colon X-\partial_S X\to X$ be the open immersion.
Then the natural transformation
\[
f^*
\xrightarrow{ad}
j_*j^*f^*
\]
is an isomorphism.
\item \textup{($\P^1$-stab)}
For $S\in \lSch/B$, $\unit_S(1)$ is $\otimes$-invertible.
\item \textup{(Loc)}
For a strict closed immersion $i$ in $\lSch/B$ with its open complement $j$,
the pair $(i^*,j^*)$ is conservative, and $i_*$ is fully faithful.
\item \textup{(Supp)}
Let
\[
\begin{tikzcd}
V\ar[d,"f'"']\ar[r,"j'"]&
X\ar[d,"f"]
\\
U\ar[r,"j"]&
S
\end{tikzcd}
\]
be a cartesian square in $\lSch/B$ such that $j$ is an open immersion and $f$ is proper.
Then the natural transformation
\[
j_\sharp f_*'
\xrightarrow{Ex}
f_*j_\sharp'
\]
is an isomorphism.
\item 
Let $(\lSch/B)_{\mathrm{comp}}$ be the subcategory of $\lSch/B$ spanned by compactifiable morphisms.
Then there exists a functor
\[
\sT_!^{ex}
\colon
(\lSch/B)_{\mathrm{comp}}
\to
\infCat_\infty
\]
such that $\sT_!^{ex}(f)\simeq f_*$ (resp.\ $\sT_!^{ex}(f)\simeq f_\sharp$) if $f$ is a proper morphism (resp.\ an open immersion).
\item \textup{(PF)}
Let $f\colon X\to S$ be a compactifiable morphism in $\lSch/B$.
Then there exists a natural isomorphism
\[
f_!\cF \otimes \cG
\simeq
f_!(\cF\otimes f^*\cG)
\]
for $\cF\in \sT^{ex}(X)$ and $\cG\in \sT^{ex}(S)$.
\end{itemize}
\end{thm}
\begin{proof}
The first 8 properties are formal consequences of the axioms of log motivic $\infty$-categories,
see \cite[Theorem 2.3.5]{logshriek}.
The last 3 properties are proved in \cite[Theorems 1.2.1, 1.2.2, 3.7.8]{logshriek}.
\end{proof}

\subsection{The result in this paper}

Our main theorem is as follows,
which establishes the Grothendieck six-functor formalism for $\sT^{ex}$ together with Theorem \ref{introduction.1}:

\begin{thm}
\label{introduction.2}
Let $\sT$ be a log motivic $\infty$-category.
Assume that $M_S(X)$ is compact in $\sT^{ex}(X)$ for every vertical exact log smooth morphism $X\to S$,
or that $\sT^{ex}$ satisfies strict \'etale descent.
Then the functor $\sT^{ex}\colon (\lSch/B)^{op}\to \CAlg(\PrL)$ satisfies the following properties.
\begin{enumerate}
\item[\textup{(1)}]
\emph{Base change.}
Let
\[
\begin{tikzcd}
X'\ar[d,"f'"']\ar[r,"g'"]&
X\ar[d,"f"]
\\
S'\ar[r,"g"]&
S
\end{tikzcd}
\]
be a cartesian square in $\lSch/B$ such that $f$ is compactifiable.
If $f$ or $g$ is $\Q$-integral,
then there exists a natural isomorphism
\[
g^*f_!
\xrightarrow{Ex}
f_!'g'^*.
\]
\item[\textup{(2)}] \emph{Poincar\'e duality.}
Let $f\colon X\to S$ be a separated vertical exact log smooth morphism in $\lSch/B$.
Then there exists a natural isomorphism
\[
f_\sharp \cF
\simeq
f_!(\cF \otimes \Thom(T_f))
\]
for $\cF\in \sT^{ex}(X)$,
where $\Thom(T_f)\in \sT^{ex}(X)$ denotes the Thom motive of the tangent bundle $T_f$ of $f$ in the sense of \textup{Definition \ref{thom.4}}.
\item[\textup{(3)}]
\emph{Lefschetz duality for a special case}. 
Let $f\colon X\to S$ be a separated log smooth morphism in $\lSch/B$.
If $S$ has the trivial log structure,
then there exists a natural isomorphism
\[
f^! \cG
\simeq
j_\sharp j^*f^*\cG
\otimes
\Thom(T_f)
\]
for $\cG\in \sT^{ex}(S)$,
where $T_f$ is the tangent bundle of $f$,
and $j\colon X-\partial X\to X$ is the obvious open immersion.
\end{enumerate}
\end{thm}
\begin{proof}
Proposition \ref{base2.5} and \cite[Corollary 3.7.6]{logshriek} show (1).
Theorem \ref{poincare.1} shows (2).
Theorem \ref{Lefschetz.4} shows (3).
\end{proof}

Note that $\infSH^{ex}$ satisfies the above compact assumption by the proof of \cite[Theorem 2.1.2]{logshriek}.

\

Let $M$ be an $n$-dimensional orientable compact manifold with boundary $\partial M$.
Assume that $\partial M=A\cup B$ for some $(n-1)$-dimensional compact submanifolds $A$ and $B$ with a common boundary $\partial A=\partial B$.
Then the classical Lefschetz duality means that there exists a canonical isomorphism
\[
H^k(M,A;\Z)
\simeq
H_k(M,B;\Z).
\]
We do not exclude the case when $A$, $B$, or $A\cap B$ is empty.
If $f\colon X\to S$ is a separated exact log smooth morphism in $\lSch/B$ with the tangent bundle $T_f$,
then there exists a natural isomorphism
\[
f_\sharp \cF \simeq f_!j_\sharp j^* (\cF \otimes \Thom(T_f))
\]
for $\cF\in \sT^{ex}(X)$ by Poincar\'e duality and ($\ver$-inv),
where $j\colon X-\partial_S X\to X$ be the obvious open immersion.
This is analogous to Lefschetz duality for the case when $A=\emptyset$.
On the other hand,
we expect that there exists a natural isomorphism
\[
f^!\cG
\simeq
j_\sharp j^* f^*\cG
\otimes
\Thom(T_f)
\]
for $\cG\in \sT^{ex}(S)$,
which is analogous to Lefschetz duality for the case when $B=\emptyset$.
Theorem \ref{Lefschetz.4} only deals with the case when $S$ has the trivial log structure.

\subsection{Chow motives over the standard log point}

For a perfect field $k$,
the category of Grothendieck's Chow motives $\Chow(k)$ is a full subcategory of $\DM(k)$,
see \cite[Remark 20.2]{MVW}.
For an fs log scheme $S$,
we define $\Chow(S)$ as the idempotent completion of the full subcategory of the homotopy category of $\Mod_{\MZ}(S)$ consisting of $M_S(X)(r)[2r]$ for every projective exact vertical log smooth morphism $X\to S$ and integer $r$.

\

The standard log point $\pt_{\N,k}$ over $k$ is the simplest log scheme over $k$ that is not a scheme.
We are primarily interested in $\Chow(\pt_{\N,k})$ since in this category,
we can use the classical arguments for Chow motives and the logarithmic arguments developed in this paper simultaneously.

\

In \cite{MR4938005},
it is suggested that the category of mixed motives over $k$ is a full subcategory of the category of pure motives over $\pt_{\N,k}$.
Such a thing should not hold at the level of Chow motives,
but the category of pure motives over $\pt_{\N,k}$ has more objects than the category of pure motives over $k$ has.
We have a similar phenomenon for Chow motives as follows.
Observe that we have $\unit(d)[n]\in \Chow(k)$ if and only if $n=2d$.

\begin{prop}[Proposition \ref{logChow.15}]
For all integers $d$ and $n$,
we have
\[
\unit(d)[n]
\in
\Chow(\pt_{\N,k}).
\]
\end{prop}

We consider a toroidal model of an elliptic curve in Example \ref{logChow.12},
which is an fs log scheme $E$ over $\pt_{\N,k}$ whose underlying scheme is the gluing of two copies of $\P^1$ along the two points $0$ and $\infty$.
When $k=\C$,
the Kato-Nakayama realization yields a map $E_{\log}\to S^1$.
Figure \ref{fig} illustrates the fiber of this map at a point,
which is homeomorphic to the topological torus.
\begin{figure}
\centering
\begin{tikzpicture}[scale = 0.75]
  \shade[ball color = gray!40, opacity = 0.4] (0,0) circle (2);
  \draw (0,0) circle (2);
  \draw (-2,0) arc (180:360:2 and 0.6);
  \draw[dashed] (2,0) arc (0:180:2 and 0.6);
  \fill[fill=black] (0,0) circle (1pt);
  \filldraw[fill=white, thick] (0,1.6) ellipse (0.4 and 0.28);
  \filldraw[fill=white, thick] (0,-1.6) ellipse (0.4 and 0.28);
  \fill[fill=black] (0.1,1.32)--(-0.1,1.24)--(-0.1,1.40);
  \node at (0,1.1) {$a$};
  \fill[fill=black] (0.1,-1.32)--(-0.1,-1.24)--(-0.1,-1.40);
  \node at (0,-1.05) {$b$};
\begin{scope}[shift={(6,0)}]
  \shade[ball color = gray!40, opacity = 0.4] (0,0) circle (2);
  \draw (0,0) circle (2);
  \draw (-2,0) arc (180:360:2 and 0.6);
  \draw[dashed] (2,0) arc (0:180:2 and 0.6);
  \fill[fill=black] (0,0) circle (1pt);
  \filldraw[fill=white, thick] (0,1.6) ellipse (0.4 and 0.28);
  \filldraw[fill=white, thick] (0,-1.6) ellipse (0.4 and 0.28);
  \fill[fill=black] (0.1,1.32)--(-0.1,1.24)--(-0.1,1.40);
  \node at (0,1.1) {$a$};
  \fill[fill=black] (0.1,-1.32)--(-0.1,-1.24)--(-0.1,-1.40);
  \node at (0,-1.05) {$b$};
\end{scope}
\end{tikzpicture}
\caption{The fiber of $E_{\log}\to S^1$}\label{fig}
\end{figure}
In Example \ref{logChow.12},
we show that there is a decomposition
\[
M_{\pt_{\N,k}}(E)
\simeq
\unit \oplus \unit[1]\oplus \unit(1)[1] \oplus \unit(1)[2].
\]
This is reminiscent of the Hodge diamond of an elliptic curve:
\[
\begin{array}{ccc}
&1
\\
1 && 1
\\
&1
\end{array}
\]

\subsection*{Acknowledgements}

This research was conducted in the framework of the DFG-funded research training group GRK 2240: \emph{Algebro-Geometric Methods in Algebra, Arithmetic and Topology}.

\subsection*{Notation and conventions}

Our standard reference for log geometry is Ogus's book \cite{Ogu}.
We employ the following notation throughout this paper:

\begin{tabular}{l|l}
$B$ & finite dimensional noetherian base scheme
\\
$\Sch/B$ & category of schemes of finite type over $B$
\\
$\lSch/B$ &  category of fs log schemes of finite type over $B$
\\
$\Sm$ & class of smooth morphisms in $\Sch/B$
\\
$\lSm$ & class of log smooth morphisms in $\lSch/B$
\\
$\eSm$ & class of exact log smooth morphisms in $\lSch/B$
\\
$\sSm$ & class of strict smooth morphisms in $\lSch/B$
\\
$\Hom_{\cC}$ & Hom space in an $\infty$-category $\cC$
\\
$\hom_{\cC}$ & Hom spectra in a stable $\infty$-category $\cC$
\\
$\infSpc$ & $\infty$-category of spaces
\\
$\sT$ & log motivic $\infty$-category
\\
$\id\xrightarrow{ad}f_*f^*$ & unit of an adjunction $(f^*,f_*)$
\\
$f^*f_*\xrightarrow{ad'}\id$ & counit of an adjunction $(f^*,f_*)$
\end{tabular}

\section{Purity}

Throughout this section,
we fix a log motivic $\infty$-category $\sT$.
The purpose of this section is to define the purity transformation
\[
\mathfrak{p}_f^n\colon f_\sharp \to f_!\Sigma_f^n
\]
for every separated exact log smooth morphism $f$,
which we do in \S \ref{puritytran}.
If $\mathfrak{p}_f^n$ is an isomorphism,
then we say that $f$ is pure.
We show in Propositions \ref{puritytran.2} and \ref{puritytran.3} that certain morphisms are pure,
but showing that every separated vertical exact log smooth morphism is pure is done in \S \ref{poincare}.

The content of \S \ref{thom}--\ref{composition} is the translation of \cite[\S 4]{logGysin} from the setting of $\sT$ to the setting of $\sT^{ex}$.
This is needed in \S \ref{naturality} to show various functorial properties of purity transformations.

\subsection{Thom transformations}
\label{thom}

Throughout this subsection, we fix a separated exact log smooth morphism $p\colon X\to S$ in $\lSch/B$ with a section $a\colon S\to X$.
Observe that $a$ is a closed immersion in the sense of \cite[Definition III.2.3.1]{Ogu}.

For $S\in \lSch/B$,
let
\[
\varphi_\sharp \colon \sT^{ex}(S)\to \sT(S)
\]
denote the inclusion functor.
Its right adjoint $\varphi^*$ exists as observed in \cite[Corollary 2.3.3]{logshriek}.
Recall from \cite[(2.3.1), (2.3.3)]{logshriek} that $\varphi_\sharp$ commutes with $f^*$ for all morphisms $f$ and $f_\sharp$ for all exact log smooth morphisms $f$.

\begin{df}
Assume that $i\colon Z\to X$ is a closed immersion in $\eSm/S$.
By \cite[Proposition 9.15]{divspc}, the kernel of the induced morphism $i^*\Omega_{X/S}^1\to \Omega_{Z/S}^1$ is a locally free $\cO_Z$-module.
The \emph{normal bundle of $Z$ in $X$ over $S$}, denoted $\Normal_{Z/S}X$, is defined to be the vector bundle associated with the dual of the kernel.
We simply use the notation $\Normal_Z X$ instead when $S$ is clear from the context.
\end{df}

\begin{df}
Let $p_n\colon \Normal_S X\to S$ be the projection, and let $a_n\colon S\to \Normal_S X$ be the $0$-section.
We set
\begin{gather*}
\Thom(p,a)
:=
p_\sharp a_*,
\;
\Thom_n(p,a)
:=
p_{n\sharp}a_{n*},
\\
\Thom_{\divi}(p,a)
:=
\varphi^*p_\sharp a_* \varphi_\sharp,
\;
\Thom_{n\divi}(p,a)
:=
\varphi^*p_{n\sharp} a_{n*}\varphi_\sharp.
\end{gather*}
\end{df}

We have the natural transformation
\begin{equation}
\label{thom.1.1}
Ex\colon \varphi_\sharp \Thom(p,a)
\to
\Thom(p,a)\varphi_\sharp
\end{equation}
given by the composition
\[
\varphi_\sharp p_\sharp a_*
\xrightarrow{\simeq}
p_\sharp \varphi_\sharp a_*
\xrightarrow{Ex}
p_\sharp a_* \varphi_\sharp.
\]
When $p=p_n$ and $a=a_n$, we have the natural transformation
\begin{equation}
\label{thom.1.2}
Ex\colon \varphi_\sharp \Thom_n(p,a)
\to
\Thom_n(p,a)\varphi_\sharp
\end{equation}

\begin{prop}
\label{thom.1}
If $a$ is a strict closed immersion, then \eqref{thom.1.1} is an isomorphism.
\end{prop}
\begin{proof}
This is an immediate consequence of
\cite[Proposition 2.4.1]{logshriek}.
\end{proof}
In particular, \eqref{thom.1.2} is always an isomorphism since $a_n$ is a strict closed immersion.

We have the natural transformation
\begin{equation}
\label{thom.2.1}
T_{\divi}
\colon
\Thom(p,a)
\to
\Thom_{\divi}(p,a)
\end{equation}
given by the composition
\[
\Thom(p,a)
\xrightarrow{ad}
\varphi^*\varphi_\sharp \Thom(p,a)
\xrightarrow{Ex}
\varphi^*\Thom(p,a)\varphi_\sharp.
\]
\begin{prop}
\label{thom.2}
If $a$ is a strict closed immersion,
then \eqref{thom.2.1} is an isomorphism.
\end{prop}
\begin{proof}
This is an immediate consequence of Proposition \ref{thom.1}.
\end{proof}

We similarly have the natural transformation
\[
T_{\divi}
\colon
\Thom_n(p,a)
\to
\Thom_{n\divi}(p,a),
\]
which is an isomorphism by Proposition \ref{thom.2} since $a_n$ is a strict closed immersion.

We apply $\varphi^*(-)\varphi_\sharp$ to the natural isomorphism $T$ in \cite[Definition 4.1.3]{logGysin}, and then we obtain a natural isomorphism
\[
T\colon \Thom_{\divi}(p,a)
\xrightarrow{\simeq}
\Thom_{n\divi}(p,a)
\]
together with \cite[Proposition 9.15]{divspc}.
Let
\[
T\colon \Thom(p,a)\to \Thom_n(p,a)
\]
be the natural transformation given by the composition
\[
\Thom(p,a) \xrightarrow{T_{\divi}}
\Thom_{\divi}(p,a)\xrightarrow{T}
\Thom_{n\divi}(p,a) \xrightarrow{T_{\divi}^{-1}}
\Thom_{n}(p,a).
\]

\begin{df}
\label{thom.4}
Let $\cE$ be a vector bundle over $S\in \lSch/B$.
The \emph{Thom motive of $\cE$} is
\[
\Thom(\cE)
:=
\cofib(M_S(\cE-Z)\to M_S(\cE))
\in
\sT^{ex}(S),
\]
where $Z$ is the zero section.
\end{df}

\begin{prop}
\label{thom.3}
Let $\cE$ be a vector bundle over $S\in \lSch/B$.
Then there is a natural isomorphism
\[
\Thom(f,i)(\cF)
\simeq
\Thom(\cE) \otimes \cF
\]
for $\cF\in \sT^{ex}(S)$,
where $f\colon \cE\to S$ is the projection,
and $i\colon S\to \cE$ is the zero section.
\end{prop}
\begin{proof}
Let $j\colon U\to \cE$ be the open complement of $a$.
By (Loc),
we have an isomorphism
$
\Thom(f,i)(\cF)
\simeq
\cofib(f_\sharp j_\sharp j^* f^* \cF\xrightarrow{ad'} f_\sharp f^* \cF).
$
We finish the proof by ($\eSm$-PF).
\end{proof}

Combine this with 
\cite[Proposition 2.4.10]{logshriek}
to see that $\Thom(\cE)$ admits a $\otimes$-inverse $\Thom(\cE)^{-1}$ for every vector bundle $\cE$ over $S\in \lSch/B$.

\subsection{Exchange transformations}

Throughout this subsection, we fix a commutative diagram in $\lSch/B$ with cartesian squares
\[
\begin{tikzcd}
S'\ar[d,"u"']\ar[r,"a'"]&
X'\ar[d,"v"]\ar[r,"p'"]&
S'\ar[d,"u"]
\\
S\ar[r,"a"]&
X\ar[r,"p"]&
S
\end{tikzcd}
\]
such that $p$ is separated exact log smooth and $pa=\id$.

\begin{const}
\label{exchange.1}
We have the natural transformation
\begin{equation}
\label{exchange.1.1}
Ex\colon u^*\Thom(p,a)\to \Thom(p',a')u^*
\end{equation}
given by the composition
\[
u^*p_\sharp a_*
\xrightarrow{Ex^{-1}}
p_\sharp' v^*a_*
\xrightarrow{Ex}
p_\sharp' a_*'u^*,
\]
where the first arrow is defined and an isomorphism by ($\eSm$-BC).
From \cite[Construction 4.2.2]{logGysin}, we have the similarly formulated natural isomorphism
\begin{equation}
\label{exchange.1.2}
Ex\colon u^*\Thom_{\divi}(p,a)\to\Thom_{\divi}(p',a')u^*.
\end{equation}
\end{const}

\begin{prop}
\label{exchange.2}
The square
\[
\begin{tikzcd}
u^*\Thom(p,a)\ar[r,"T_{\divi}"]\ar[d,"Ex"']&
u^*\Thom_{\divi}(p,a)\ar[d,"Ex"]
\\
\Thom(p',a')u^*\ar[r,"T_{\divi}"]&
\Thom_{\divi}(p',a')u^*
\end{tikzcd}
\]
commutes.
\end{prop}
\begin{proof}
The diagram
\[
\begin{tikzcd}[row sep=small]
u^*p_\sharp a_*\ar[ddd,"Ex^{-1}"']\ar[r,"ad"]\ar[rdd,"ad"',bend right=20]&
u^*\varphi^*\varphi_\sharp p_{\sharp}a_*\ar[d,"Ex"]\ar[r,"\simeq"]&
u^*\varphi^*p_{\sharp}\varphi_\sharp a_*\ar[d,"Ex"]\ar[r,"Ex"]&
u^*\varphi^*p_{\sharp}a_*\varphi_\sharp\ar[d,"Ex"]
\\
&
\varphi^*u^*\varphi_\sharp p_{\sharp}a_*\ar[d,"\simeq"]\ar[r,"\simeq"]&
\varphi^*u^*p_{\sharp}\varphi_\sharp a_*\ar[dd,"Ex^{-1}"]\ar[r,"Ex"]&
\varphi^*u^*p_{\sharp}a_*\varphi_\sharp\ar[dd,"Ex^{-1}"]
\\
&
\varphi^*\varphi_\sharp u^*p_{\sharp}a_*\ar[d,"Ex^{-1}"]
\\
p_{\sharp}'v^*a_*\ar[r,"ad"]\ar[dd,"Ex"']&
\varphi^*\varphi_\sharp p_{\sharp}'v^*a_*\ar[r,"\simeq"]\ar[dd,"Ex"]&
\varphi^*p_{\sharp}'v^*\varphi_\sharp a_*\ar[r,"Ex"]\ar[d,"\simeq"]&
\varphi^*p_{\sharp}'v^*a_*\varphi_\sharp \ar[d,"Ex"]
\\
&
&
\varphi^*p_{\sharp}'\varphi_\sharp v^*a_*\ar[d,"Ex"]&
\varphi^*p_{\sharp}'a_*'u^*\varphi_\sharp\ar[d,"\simeq"]
\\
p_{\sharp}'a_*'u^*\ar[r,"ad"]&
\varphi^*\varphi_\sharp p_{\sharp}'a_*'u^*\ar[r,"\simeq"]&
\varphi^* p_{\sharp}'\varphi_\sharp a_*'u^*\ar[r,"Ex"]&
\varphi^* p_{\sharp}' a_*'\varphi_\sharp u^*
\end{tikzcd}
\]
commutes, which shows the claim.
\end{proof}

\begin{const}
\label{exchange.3}
We have the natural transformation
\begin{equation}
\label{exchange.3.1}
Ex\colon
\Thom(p,a)u_*
\to
u_*\Thom(p',a')
\end{equation}
given by the composition
\[
p_\sharp a_* u_*
\xrightarrow{\simeq}
p_\sharp v_* a_*'
\xrightarrow{Ex}
u_*p_\sharp' a_*'.
\]
From \cite[Construction 4.2.3]{logGysin}, we have the similarly formulated natural transformation
\begin{equation}
\label{exchange.3.2}
Ex\colon \Thom_{\divi}(p,a)u_*
\to
u_*\Thom_{\divi}(p',a').
\end{equation}
\end{const}

\begin{prop}
\label{exchange.4}
The square
\[
\begin{tikzcd}
\Thom(p,a)u_*\ar[r,"T_{\divi}"]\ar[d,"Ex"']&
\Thom_{\divi}(p,a)u_*\ar[d,"Ex"]
\\
u_*\Thom(p',a')\ar[r,"T_{\divi}"]&
u_*\Thom_{\divi}(p',a')
\end{tikzcd}
\]
commutes.
\end{prop}
\begin{proof}
The diagram
\[
\begin{tikzcd}[row sep=small]
p_\sharp a_*u_*\ar[r,"ad"]\ar[d,"\simeq"']&
\varphi^*\varphi_\sharp p_\sharp a_*u_*\ar[r,"\simeq"]\ar[d,"\simeq"]&
\varphi^*p_\sharp \varphi_\sharp a_*u_*\ar[r,"Ex"]\ar[d,"\simeq"]&
\varphi^*p_\sharp a_*\varphi_\sharp u_*\ar[d,"Ex"]
\\
p_\sharp v_*a_*'\ar[ddd,"Ex"']\ar[r,"ad"]&
\varphi^*\varphi_\sharp p_\sharp v_*a_*'\ar[d,"Ex"]\ar[r,"\simeq"]&
\varphi^*p_\sharp \varphi_\sharp v_*a_*'\ar[d,"Ex"]&
\varphi^*p_\sharp a_* u_*\varphi_\sharp\ar[d,"\simeq"]
\\
&
\varphi^*\varphi_\sharp u_*p_\sharp' a_*'\ar[d,"Ex"]&
\varphi^*p_\sharp  v_*\varphi_\sharp a_*'\ar[r,"Ex"]\ar[d,"Ex"]&
\varphi^*p_\sharp v_*a_*'\varphi_\sharp\ar[d,"Ex"]
\\
&
\varphi^* u_*\varphi_\sharp p_\sharp' a_*'\ar[r,"\simeq"]\ar[d,"\simeq"]&
\varphi^* u_*p_\sharp' \varphi_\sharp  a_*'\ar[r,"Ex"]\ar[d,"\simeq"]&
\varphi^*u_*p_\sharp'a_*'\varphi_\sharp\ar[d,"\simeq"]
\\
u_*p_\sharp'a_*'\ar[r,"ad"]\ar[ruu,"ad",bend left=20]&
u_*\varphi^*\varphi_\sharp p_\sharp'a_*'\ar[r,"\simeq"]&
u_*\varphi^*p_\sharp'\varphi_\sharp a_*'\ar[r,"Ex"]&
u_*\varphi^*p_\sharp'a_*'\varphi_\sharp
\end{tikzcd}
\]
commutes, which shows the claim.
\end{proof}

\begin{const}
\label{exchange.5}
If $u$ is exact log smooth, then we have the natural transformation
\begin{equation}
\label{exchange.5.1}
Ex\colon u_\sharp \Thom(p',a') \to \Thom(p,a)u_\sharp
\end{equation}
given by the composition
\[
u_\sharp p_\sharp' a_*'
\xrightarrow{\simeq}
p_\sharp v_\sharp a_*'
\xrightarrow{Ex}
p_\sharp a_* u_\sharp.
\]
From \cite[Construction 4.2.1]{logGysin}, we have the similarly formulated natural transformation
\begin{equation}
\label{exchange.5.2}
Ex\colon  u_\sharp \Thom_{\divi}(p',a') \to \Thom_{n\divi}(p,a)u_\sharp.
\end{equation}
\end{const}

\begin{prop}
\label{exchange.6}
If $u$ is exact log smooth, then the square
\[
\begin{tikzcd}
u_\sharp \Thom(p'.a')\ar[d,"Ex"']\ar[r,"T_{\divi}"]&
u_\sharp \Thom_{\divi}(p',a')\ar[d,"Ex"]&
\\
\Thom(p,a)u_\sharp \ar[r,"T_{\divi}"]&
\Thom_{\divi}(p,a)u_\sharp
\end{tikzcd}
\]
commutes.
\end{prop}
\begin{proof}
The diagram
\[
\begin{tikzcd}[column sep=small, row sep=small]
u_\sharp p_\sharp'a_*'\ar[rr,"ad"]\ar[rd,"ad"']\ar[dd,"\simeq"']&
&
u_\sharp \varphi^*\varphi_\sharp p_\sharp' a_*'\ar[d,"Ex"]\ar[r,"\simeq"]&
u_\sharp \varphi^* p_\sharp' \varphi_\sharp a_*'\ar[d,"Ex"]\ar[r,"Ex"]&
u_\sharp \varphi^* p_\sharp' a_*' \varphi_\sharp\ar[d,"Ex"]
\\
&
\varphi^* \varphi_\sharp u_\sharp p_\sharp' a_*'\ar[d,"\simeq"]\ar[r,"\simeq"]&
\varphi^* u_\sharp \varphi_\sharp p_\sharp' a_*'\ar[d,"\simeq"]\ar[r,"\simeq"]&
\varphi^*u_\sharp p_\sharp' \varphi_\sharp a_*'\ar[d,"\simeq"]\ar[r,"Ex"]&
\varphi^*u_\sharp p_\sharp'a_*'\varphi_\sharp\ar[d,"\simeq"]
\\
p_\sharp v_\sharp a_*'\ar[r,"ad"]\ar[dd,"Ex"']&
\varphi^*\varphi_\sharp p_\sharp v_\sharp a_*'\ar[dd,"Ex"]\ar[r,"\simeq"]&
\varphi^*p_\sharp \varphi_\sharp v_\sharp a_*'\ar[dd,"Ex"]\ar[r,"\simeq"]&
\varphi^*p_\sharp v_\sharp \varphi_\sharp a_*'\ar[r,"Ex"]&
\varphi^*p_\sharp v_\sharp a_*' \varphi_\sharp\ar[d,"Ex"]
\\
&
&
&
&
\varphi^*p_\sharp a_* u_\sharp \varphi_\sharp\ar[d,"\simeq"]
\\
p_\sharp a_*u_\sharp\ar[r,"ad"]&
\varphi^*\varphi_\sharp p_\sharp a_*u_\sharp\ar[r,"\simeq"]&
\varphi^*p_\sharp \varphi_\sharp a_*u_\sharp\ar[rr,"Ex"]&
&
\varphi^*p_\sharp a_*\varphi_\sharp u_\sharp
\end{tikzcd}
\]
commutes, which shows the claim.
\end{proof}

\begin{prop}
\label{exchange.10}
Assume that $a$ is strict.
Then the natural transformations \eqref{exchange.1.1} and \eqref{exchange.3.1} are isomorphisms.
If we assume further that $u$ is exact log smooth, then the natural transformation \eqref{exchange.5.1} is an isomorphism.
\end{prop}
\begin{proof}
The natural transformations \eqref{exchange.3.2} and \eqref{exchange.5.2} are isomorphisms by \cite[Proposition 4.2.7, 4.2.9]{logGysin}.
Proposition \ref{thom.1} implies that the natural transformations
\[
T_{\divi}
\colon
\Thom(p,a)\to \Thom_{\divi}(p,a)
\text{ and }
T_{\divi}
\colon
\Thom(p',a')\to \Thom_{\divi}(p',a')
\]
are isomorphisms.
Use Propositions \ref{exchange.2}, \ref{exchange.4}, and \ref{exchange.6} to conclude.
\end{proof}

\begin{prop}
\label{exchange.7}
The squares
\[
\begin{tikzcd}
u^*\Thom(p,a)\ar[r,"T"]\ar[d,"Ex"']&
u^*\Thom_{n}(p,a)\ar[d,"Ex"]
\\
\Thom(p',a')u^*\ar[r,"T"]&
\Thom_{n}(p',a')u^*
\end{tikzcd}
\;\;
\begin{tikzcd}
\Thom(p,a)u_*\ar[r,"T"]\ar[d,"Ex"']&
\Thom_n(p,a)u_*\ar[d,"Ex"]
\\
u_*\Thom(p',a')\ar[r,"T"]&
u_*\Thom_{n}(p',a')
\end{tikzcd}
\]
commute.
If $u$ is exact log smooth, then the square
\[
\begin{tikzcd}
u_\sharp \Thom(p'.a')\ar[d,"Ex"']\ar[r,"T"]&
u_\sharp \Thom_{n}(p',a')\ar[d,"Ex"]&
\\
\Thom(p,a)u_\sharp \ar[r,"T"]&
\Thom_{n}(p,a)u_\sharp
\end{tikzcd}
\]
commute.
\end{prop}
\begin{proof}
The first square commutes since the diagram
\[
\begin{tikzcd}
u^*\Thom(p,a)\ar[r,"T_{\divi}"]\ar[d,"Ex"']&
u^*\Thom_{\divi}(p,a)\ar[r,"T"]\ar[d,"Ex"]&
u^*\Thom_{n\divi}(p,a)\ar[r,"T_{\divi}^{-1}"]\ar[d,"Ex"]&
u^*\Thom_{n}(p,a)\ar[d,"Ex"]
\\
\Thom(p',a')u^*\ar[r,"T_{\divi}"]&
\Thom_{\divi}(p',a')u^*\ar[r,"T"]&
\Thom_{n\divi}(p',a')u^*\ar[r,"T_{\divi}^{-1}"]&
\Thom_{n}(p',a')u^*
\end{tikzcd}
\]
commutes by \cite[Proposition 4.2.12]{logGysin} and Proposition \ref{exchange.2}.
We can similarly show that the second (resp.\ third) square commutes by \cite[Proposition 4.2.13]{logGysin} and Proposition \ref{exchange.4} (resp.\ \cite[Proposition 4.2.11]{logGysin} and Proposition \ref{exchange.6}).
\end{proof}

\begin{const}
\label{exchange.8}
Assume that $u$ is compactifiable.
By
\cite[Construction 3.7.3]{logshriek},
$u$ admits a factorization
\[
S'\xrightarrow{u'}
S''
\xrightarrow{u''}
S'''
\xrightarrow{u'''}
S
\]
such that $\ul{u'}$ is an isomorphism, $u''$ is an open immersion, and $u'''$ is proper.
Let $p''$ and $a''$ (resp.\ $p'''$ and $a'''$) be the pullbacks of $p$ and $a$ along $S''\to S$ (resp.\ $S'''\to S$).
We have the natural transformation
\[
Ex\colon
\Thom(p,a)u_!
\to
u_!\Thom(p',a')
\]
given by the composition
\begin{align*}
\Thom(p,a)u_*'''u_\sharp''u_*'
\xrightarrow{Ex} &
u_*'''\Thom(p''',a''')u_\sharp''u_*'
\\
\xrightarrow{Ex^{-1}}&
u_*'''u_\sharp''\Thom(p'',a'')u_*'
\xrightarrow{Ex}
u_*'''u_\sharp''u_*'\Thom(p',a'),
\end{align*}
where the second arrow is defined and an isomorphism by Proposition \ref{exchange.10}.
\end{const}

\begin{prop}
\label{exchange.9}
If $u$ is compactifiable,
then the square
\[
\begin{tikzcd}
\Thom(p,a)u_!\ar[r,"T"]\ar[d,"Ex"']&
\Thom_n(p,a)u_!\ar[d,"Ex"]
\\
u_!\Thom(p',a')\ar[r,"T"]&
u_!\Thom_{n}(p',a')
\end{tikzcd}
\]
commutes.
\end{prop}
\begin{proof}
This is an immediate consequence of Proposition \ref{exchange.7}.
\end{proof}

\subsection{Composition transformations}
\label{composition}

Throughout this subsection, we fix a commutative diagram in $\lSch/B$
\[
\begin{tikzcd}
S\ar[d,"c"']\ar[rd,"b"]
\\
Z\ar[d,"r"']\ar[r,"g"]&
Y\ar[d,"f"]\ar[rd,"q"]
\\
S\ar[r,"a"]&
X\ar[r,"p"]&
S
\end{tikzcd}
\]
such that the inner square is cartesian, $f,p,q,r\in \eSm$, $pa=\id$, $qb=\id$, and $rc=\id$.
We have the natural transformation
\begin{equation}
\label{comp.4.1}
C
\colon
\Thom(q,b)
\to
\Thom(p,a)\Thom(r,c)
\end{equation}
given by the composition
\[
q_\sharp b_*
\xrightarrow{\simeq}
p_\sharp f_\sharp g_* c_*
\xrightarrow{Ex}
p_\sharp a_*r_\sharp c_*.
\]

\begin{prop}
\label{comp.4}
If $f$ is strict, then \eqref{comp.4.1} is an isomorphism.
\end{prop}
\begin{proof}
This is an immediate consequence of \cite[Proposition 2.4.9]{logshriek}.
\end{proof}

\begin{prop}
\label{comp.5}
There is a natural isomorphism
\[
\Thom(\A_S^n)
\simeq
\unit(n)[2n]
\]
in $\sT^{ex}(S)$ for every integer $n\geq 0$.
\end{prop}
\begin{proof}
The claim is trivial if $n=0,1$.
For $n\geq 2$,
use induction and Propositions \ref{thom.3} and \ref{comp.4}.
\end{proof}

\begin{prop}
\label{comp.1}
The square
\[
\begin{tikzcd}
\varphi_\sharp\Thom(q,b)\ar[rr,"Ex"]\ar[d,"C"']&
&
\Thom(q,b)\varphi_\sharp\ar[d,"C"]
\\
\varphi_\sharp \Thom(r,c)\Thom(p,a)\ar[r,"Ex"]&
\Thom(r,c)\varphi_\sharp \Thom(p,a)\ar[r,"Ex"]&
\Thom(r,c)\Thom(p,a)\varphi_\sharp
\end{tikzcd}
\]
commutes,
where the right vertical arrow is obtained by \cite[\S 4.3]{logGysin}.
\end{prop}
\begin{proof}
The diagram
\[
\begin{tikzcd}[column sep=small, row sep=small]
\varphi_\sharp q_\sharp b_*\ar[rr,"\simeq"]\ar[d,"\simeq"']&
&
q_\sharp \varphi_\sharp b_*\ar[rr,"Ex"]\ar[d,"\simeq"']&
&
q_\sharp b_*\varphi_\sharp\ar[d,"\simeq"]
\\
\varphi_\sharp p_\sharp f_\sharp g_*c_*\ar[r,"\simeq"]\ar[d,"Ex"]&
p_\sharp \varphi_\sharp f_\sharp g_*c_*\ar[r,"\simeq"]\ar[d,"Ex"]&
p_\sharp f_\sharp \varphi_\sharp g_*c_*\ar[r,"Ex"]&
p_\sharp f_\sharp g_*\varphi_\sharp c_*\ar[d,"Ex"]\ar[r,"Ex"]&
p_\sharp f_\sharp g_*c_*\varphi_\sharp\ar[d,"Ex"]
\\
\varphi_\sharp p_\sharp a_* r_\sharp c_*\ar[r,"\simeq"]&
p_\sharp \varphi_\sharp a_*r_\sharp c_*\ar[r,"Ex"]&
p_\sharp a_*\varphi_\sharp r_\sharp c_*\ar[r,"\simeq"]&
p_\sharp a_* r_\sharp \varphi_\sharp c_*\ar[r,"Ex"]&
p_\sharp a_* r_\sharp c_* \varphi_\sharp
\end{tikzcd}
\]
commutes, which shows the claim.
\end{proof}

From \cite[\S 4.3]{logGysin}, we have the natural transformation
\[
C\colon
\Thom_{\divi}(q,b)
=
\varphi^*\Thom(q,b)\varphi_\sharp
\to
\varphi^*\Thom(r,c)\Thom(p,a)\varphi_\sharp.
\]
We consider the composition
\[
T_{\divi,C}
\colon
\Thom(r,c)\Thom(p,a)
\xrightarrow{ad}
\varphi^*\varphi_\sharp\Thom(r,c)\Thom(p,a)
\xrightarrow{ExEx}
\varphi^*\Thom(r,c)\Thom(p,a)\varphi_\sharp.
\]

\begin{prop}
\label{comp.2}
The square
\[
\begin{tikzcd}
\Thom(q,b)\ar[d,"C"']\ar[r,"T_{\divi}"]&
\Thom_{\divi}(q,b)\ar[d,"C"]
\\
\Thom(r,c)\Thom(p,a)\ar[r,"T_{\divi,C}"]&
\varphi^*\Thom(r,c)\Thom(p,a)\varphi_\sharp
\end{tikzcd}
\]
commutes.
\end{prop}
\begin{proof}
It suffices to check that the diagram
\[
\begin{tikzcd}
\Thom(q,b)\ar[d,"C"']\ar[r,"ad"]&
\varphi^*\varphi_\sharp \Thom(q,b)\ar[d,"C"]\ar[r,"Ex"]&
\varphi^*\Thom(q,b)\varphi_\sharp\ar[d,"C"]
\\
\Thom(r,c)\Thom(p,a)\ar[r,"ad"]&
\varphi^*\varphi_\sharp \Thom(r,c)\Thom(p,a)\ar[r,"ExEx"]&
\varphi^*\Thom(r,c)\Thom(p,a)\varphi_\sharp
\end{tikzcd}
\]
commutes.
The left square trivially commutes.
The right square commutes by Proposition \ref{comp.1}.
\end{proof}

By \cite[Theorem IV.3.2.3]{Ogu}, we have the exact sequence
\[
0\to f^*\Omega_{X/S}^1 \to \Omega_{Y/S}^1 \to \Omega_{Y/X}^1 \to 0.
\]
Apply $b^*$ to this sequence and use \cite[Proposition IV.1.2.15(2)]{Ogu} to obtain an exact sequence
\[
0\to a^*\Omega_{X/S}^1 \to b^*\Omega_{Y/S}^1 \to c^*\Omega_{Z/S}^1 \to 0.
\]
Together with \cite[Theorem IV.3.2.2]{Ogu}, we see that the inner square of the induced commutative diagram
\[
\begin{tikzcd}
S\ar[d,"c_n"']\ar[rd,"b_n"]
\\
\Normal_S Z\ar[d,"r_n"']\ar[r,"g_n"]&
\Normal_S Y\ar[d,"f_n"]\ar[rd,"q_n"]
\\
S\ar[r,"a_n"]&
\Normal_S X\ar[r,"p_n"]&
S
\end{tikzcd}
\]
is cartesian.
Hence we have the similarly obtained natural transformations
\begin{gather*}
C\colon \Thom_n(q,b)
\to
\Thom_n(r,c)\Thom_n(p,a),
\\
C\colon \Thom_{n\divi}(q,b)
\to
\varphi^*\Thom_n(r,c)\Thom_n(p,a)\varphi_\sharp,
\\
T_{\divi,C}
\colon
\Thom_n(r,c)\Thom_n(p,a)
\to
\varphi^*\Thom_n(r,c)\Thom_n(p,a)\varphi_\sharp.
\end{gather*}
The last one is an isomorphism by Proposition \ref{thom.1}.

\begin{prop}
\label{comp.3}
The square
\[
\begin{tikzcd}
\Thom(q,b)\ar[d,"C"']\ar[r,"T"]&
\Thom_n(q,b)\ar[d,"C"]
\\
\Thom(r,c)\Thom(p,a)\ar[r,"T"]&
\Thom_n(r,c)\Thom_n(p,a)
\end{tikzcd}
\]
commutes.
\end{prop}
\begin{proof}
It suffices to check that the diagram
\[
\begin{tikzcd}[row sep=small]
\Thom(q,b)\ar[r,"C"]\ar[d,"T_{\divi}"']&
\Thom(r,c)\Thom(p,a)\ar[d,"T_{\divi,C}"]
\\
\Thom_{\divi}(q,b)\ar[d,"T"']\ar[r,"C"]&
\varphi^*\Thom(r,c)\Thom(p,a)\varphi_\sharp \ar[d,"TT"]
\\
\Thom_{n\divi}(q,b)\ar[d,"T_{\divi}^{-1}"']\ar[r,"C"]&
\varphi^*\Thom_n(r,c)\Thom_n(p,a)\varphi_\sharp
\ar[d,"T_{\divi,C}^{-1}"]
\\
\Thom_n(q,b)\ar[r,"C"]&
\Thom_n(r,c)\Thom_n(p,a)
\end{tikzcd}
\]
commutes.
The top and bottom squares commute by Proposition \ref{comp.2}.
The middle square commutes by \cite[Proposition 4.3.1]{logGysin}.
\end{proof}

\subsection{Purity transformations}
\label{puritytran}

Throughout this subsection, we fix a separated exact log smooth morphism $f\colon X\to S$ in $\lSch/B$ and the commutative diagram
\[
\begin{tikzcd}
X\ar[r,"a"]&
X\times_S X\ar[d,"p_1"']\ar[r,"p_2"]&
X\ar[d,"f"]
\\
&
X\ar[r,"f"]&
S,
\end{tikzcd}
\]
where $p_1$ (resp.\ $p_2$) is the first (resp.\ second) projection, and $a$ is the diagonal morphism.
Let $p_{2n}\colon \Normal_X(X\times_S X)\to X$ be the projection, and let $a_n\colon X\to \Normal_X(X\times_S X)$ be the $0$-section.

\begin{df}
We set
\[
\Sigma_f:=\Thom(p_2,a),
\;
\Sigma_f^n:=\Thom(p_{2n},a_n),
\;
\Sigma_f^{\divi}:=\Thom_{\divi}(p_2,a),
\;
\Sigma_f^{n\divi}:=\Thom_{n\divi}(p_2,a).
\]
Let $\Omega_f$, $\Omega_f^n$, $\Omega_f^{\divi}$, and $\Omega_f^{n\divi}$ be their right adjoints.
We have the natural transformations
\[
\Sigma_f\xrightarrow{T_{\divi}}
\Sigma_f^{\divi}\xrightarrow{T}
\Sigma_f^{n\divi}\xrightarrow{T_{\divi}^{-1}}
\Sigma_f^n
\]
whose composition is $T\colon \Sigma_f \to \Sigma_f^n$.
\end{df}

\begin{df}
We have the natural transformation
\[
\mathfrak{p}_f
\colon
f_\sharp \to f_!\Sigma_f
\]
for $\sT^{ex}$ given by the composition
\[
f_\sharp
\xrightarrow{\simeq}
f_\sharp p_{1!}a_*
\xrightarrow{Ex}
f_! p_{2\sharp} a_*
=
f_! \Sigma_f.
\]
Compose $\mathfrak{p}_f$ with $T_{\divi}\colon \Sigma_f\to \Sigma_f^{\divi}$ and $T\colon \Sigma_f \to \Sigma_f^n$ to obtain the natural transformations
\[
\mathfrak{p}_f^{\divi}
\colon
f_\sharp\to f_!\Sigma_f^{\divi}
\text{ and }
\mathfrak{p}_f^{n}
\colon
f_\sharp\to f_!\Sigma_f^{n}.
\]
We are using the Nagata compactification theorem here to use $f_!$.

We say that $f$ is \emph{pure} if $\mathfrak{p}_f^{\divi}$ is an isomorphism, or equivalently $\mathfrak{p}_f^n$ is an isomorphism.
Let
\[
\mathfrak{q}_f
\colon
\Omega_f f^! \to f^*,
\;
\mathfrak{q}_f^{\divi}
\colon
\Omega_f^{\divi} f^!\to f^*,
\text{ and }
\mathfrak{q}_f^{n}
\colon
\Omega_f^n f^! \to f^*
\]
be right adjoints of $\mathfrak{p}_f$, $\mathfrak{p}_f^{\divi}$, and $\mathfrak{p}_f^n$.
\end{df}

\begin{prop}
\label{puritytran.2}
The functor $\Sigma_f^n$ is an equivalence of $\infty$-categories.
\end{prop}
\begin{proof}
This is an immediate consequence of
\cite[Proposition 2.4.10]{logshriek}.
\end{proof}

\begin{prop}
\label{puritytran.1}
Let $f\colon X\to S$ be a strict separated smooth morphism in $\lSch/B$.
Then $\mathfrak{p}_f\colon f_\sharp \to f_!\Sigma_f$ is an isomorphism,
and $f$ is pure.
\end{prop}
\begin{proof}
Apply \cite[Theorem 2.4.26(3)]{CD19} to $\Ho(\ul{\sT}_{/S}^{ex})$ in 
\cite[Definition 2.4.8]{logshriek}
to see that the natural transformation $\mathfrak{p}_f\colon f_\sharp \to f_!\Sigma_f$ is an isomorphism.
To conclude,
observe that the natural transformation $T_{\divi}\colon \Sigma_f\to \Sigma_f^{\divi}$ is an isomorphism by
\cite[Proposition 2.4.1]{logshriek}.
\end{proof}

\begin{prop}
\label{puritytran.3}
Let $W$ be the gluing of
\[
\Spec(\N x\oplus \N y\to \Z[x,y]),
\;
\Spec(\N(xy)\to \Z[xy,y^{-1}]),
\;
\Spec(\N(xy)\to \Z[xy,x^{-1}]).
\]
Consider the morphism $W\to \A_\N$ of fs log schemes induced by the diagonal homomorphism $\N\to \N\oplus \N$.
Then for every $S\in \lSch/B$ with a morphism $S\to \A_\N$ of fs log schemes,
the projection $f\colon S\times_{\A_\N} W\to S$ is pure.
\end{prop}
\begin{proof}
The diagram
\begin{equation}
\label{puritytran.3.1}
\begin{tikzcd}
\varphi_\sharp f_\sharp \ar[rr,"\simeq"]\ar[d,"\mathfrak{p}_f"']&
&
f_\sharp \varphi_\sharp\ar[d,"\mathfrak{p}_f"]
\\
\varphi_\sharp f_* \Sigma_f \ar[r,"Ex"]&
f_* \varphi_\sharp \Sigma_f \ar[r,"Ex"]&
f_* \Sigma_f \varphi_\sharp
\end{tikzcd}
\end{equation}
commutes since the diagram
\[
\begin{tikzcd}[row sep=small, column sep=small]
\varphi_\sharp f_\sharp
\ar[rrr,"\simeq"]\ar[d,"\simeq"']
&
&
&
f_\sharp \varphi_\sharp
\ar[d,"\simeq"]
\\
\varphi_\sharp f_\sharp p_{1*}a_*
\ar[d,"Ex"']\ar[r,"\simeq"]
&
f_\sharp \varphi_\sharp p_{1*}a_*
\ar[r,"Ex"]
&
f_\sharp p_{1*}\varphi_\sharp a_*
\ar[r,"Ex"]\ar[d,"Ex"]
&
f_\sharp p_{1*} a_* \varphi_\sharp
\ar[d,"Ex"]
\\
\varphi_\sharp f_* p_{2\sharp} a_*
\ar[r,"Ex"]
&
f_* \varphi_\sharp p_{2\sharp} a_*
\ar[r,"\simeq"]
&
f_* p_{2\sharp} \varphi_\sharp a_*
\ar[r,"Ex"]
&
f_* p_{2\sharp} a_* \varphi_\sharp
\end{tikzcd}
\]
commutes.
The right vertical arrow of \eqref{puritytran.3.1} is an isomorphism by \cite[Theorem 5.2.11]{logGysin},
so the composite $\varphi_\sharp f_\sharp \to f_*\Sigma_f \varphi_\sharp$ is an isomorphism.
This implies that the composite $f_\sharp \to \varphi^* f_* \Sigma_f \varphi_\sharp$ in the commutative diagram
\[
\begin{tikzcd}
f_\sharp\ar[d,"ad"']\ar[r,"\mathfrak{p}_f"]&
f_*\Sigma_f \ar[d,"ad"]\ar[r,"ad"]&
f_*\varphi^*\varphi_\sharp \Sigma_f\ar[d,"\simeq",leftarrow]\ar[r,"Ex"]&
f_*\varphi^*\Sigma_f \varphi_\sharp\ar[d,"\simeq",leftarrow]
\\
\varphi^* \varphi_\sharp f_\sharp\ar[r,"\mathfrak{p}_f"]&
\varphi^* \varphi_\sharp f_* \Sigma_f \ar[r,"Ex"]&
\varphi^* f_* \varphi_\sharp \Sigma_f\ar[r,"Ex"]&
\varphi^*f_*\Sigma_f \varphi_\sharp
\end{tikzcd}
\]
is an isomorphism since $\varphi_\sharp$ is fully faithful.
It follows that the composite $f_\sharp\to f_*\varphi^*\Sigma_f \varphi_\sharp$ is an isomorphism,
which is $\mathfrak{p}_f^{\divi}$.
\end{proof}

\begin{prop}
\label{Verdier.9}
Assume that $M_S(U)$ is compact in $\sT^{ex}(S)$ for every vertical exact log smooth morphism $U\to S$ in $\lSch/B$.
Let $f\colon Y\to X$ be a morphism in $\lSch/B$.
Then $f_*$ and $f_!$ preserve colimits.
\end{prop}
\begin{proof}
We only need to show that $f_*$ preserves colimits.
By \cite[Proposition 1.4.4.1(2)]{HA},
it suffices to show that $f_*$ preserves sums.
For all vertical exact log smooth morphism $V\to X$ in $\lSch/B$ and integer $d$,
it suffices to show that the functor
\[
\Hom_{\sT^{ex}(X)}(M_X(V)(d),f_*(-))
\]
preserves sums.
This holds since $f^*M_X(V)(d)$ is compact by assumption.
\end{proof}

\begin{df}
\label{Verdier.8}
For a separated exact log smooth morphism $f\colon X\to S$ in $\lSch/B$ such that $f_!$ preserves colimits,
let $f^!$ be a right adjoint of $f_!$.
In this case,
let
$
\mathfrak{q}_f
\colon
\Omega_f f^! \to f^*$,
$
\mathfrak{q}_f^{\divi}
\colon
\Omega_f^{\divi} f^!\to f^*
$,
and
$
\mathfrak{q}_f^{n}
\colon
\Omega_f^n f^! \to f^*
$
be right adjoints of $\mathfrak{p}_f$, $\mathfrak{p}_f^{\divi}$, and $\mathfrak{p}_f^n$.
\end{df}

\subsection{Naturality of the purity transformations}
\label{naturality}

\begin{prop}
\label{functorial.1}
Let
\[
\begin{tikzcd}
X'\ar[r,"g'"]\ar[d,"f'"']&
X\ar[d,"f"]
\\
S'\ar[r,"g"]&
S
\end{tikzcd}
\]
be a cartesian square in $\lSch/B$ such that $f$ is separated vertical exact log smooth.
Then the diagram
\begin{equation}
\label{functorial.1.1}
\begin{tikzcd}
f_\sharp'g'^*\ar[rr,"Ex"]\ar[d,"\mathfrak{p}_{f'}^n"']&
&
g^*f_\sharp \ar[d,"\mathfrak{p}_f^n"]
\\
f_!'\Sigma_f^n g'^*\ar[r,"Ex",leftarrow]&
f_!'g'^*\Sigma_f^n\ar[r,"Ex",leftarrow]&
g^*f_!\Sigma_f^n
\end{tikzcd}
\end{equation}
commutes.
\end{prop}
\begin{proof}
It suffices to check that the diagram
\[
\begin{tikzcd}[row sep=small]
f_\sharp'g'^*\ar[rr,"Ex"]\ar[d,"\mathfrak{p}_{f'}"']\ar[rrd,"(a)" description,phantom]&
&
g^*f_\sharp \ar[d,"\mathfrak{p}_f"]
\\
f_!'\Sigma_f g'^*\ar[r,"Ex",leftarrow]\ar[d,"T"']\ar[rd,"(b)" description,phantom]&
f_!'g'^*\Sigma_f\ar[r,"Ex",leftarrow]\ar[d,"T"]\ar[rd,"(c)" description,phantom]&
g^*f_!\Sigma_f\ar[d,"T"]
\\
f_!'\Sigma_f^n g'^*\ar[r,"Ex",leftarrow]&
f_!'g'^*\Sigma_f^n\ar[r,"Ex",leftarrow]&
g^*f_!\Sigma_f^n
\end{tikzcd}
\]
commutes.
The square (b) commutes by Proposition \ref{exchange.7}.
The square (c) trivially commutes.
The diagram
\cite[(2.4.2)]{logshriek} commutes,
and one can similarly show that the diagram surrounding $(a)$ commutes.
\end{proof}

\begin{rmk}
\label{functorial.8}
In \eqref{functorial.1.1},
the natural transformations $f_\sharp' g'^*\xrightarrow{Ex} g^*f_\sharp$ and $f_!'g'^*\Sigma_f^n\to f_!'\Sigma_f^n g'^*$ are isomorphisms by ($\eSm$-BC) and Proposition \ref{exchange.10}.
\end{rmk}

\begin{prop}
\label{functorial.7}
Let
\[
\begin{tikzcd}
X'\ar[r,"g'"]\ar[d,"f'"']&
X\ar[d,"f"]
\\
S'\ar[r,"g"]&
S
\end{tikzcd}
\]
be a cartesian square in $\lSch/B$ such that $f$ is separated vertical exact log smooth.
If $f$ is pure,
then the natural transformation
\[
Ex\colon
g^*f_!
\to
f_!'g'^*
\]
is an isomorphism.
\end{prop}
\begin{proof}
By Proposition \ref{functorial.1} and Remark \ref{functorial.8},
the natural transformation $g^*f_!\Sigma_f^n \to f_!'g'^*\Sigma_f^n$ is an isomorphism.
Proposition \ref{puritytran.2} finishes the proof.
\end{proof}

\begin{prop}
\label{functorial.3}
Let $Y\xrightarrow{g} X\xrightarrow{f} S$ be separated exact log smooth morphisms in $\lSch/B$.
Then the diagram
\begin{equation}
\label{functorial.3.1}
\begin{tikzcd}[row sep=small]
f_\sharp g_\sharp\ar[dd,"\simeq"']\ar[rr,"\mathfrak{p}_f^n \mathfrak{p}_g^n"]&
&
f_!\Sigma_f^n g_!\Sigma_g^n\ar[d,"Ex"]
\\
&
&
f_!g_!\Thom_n(r_2,c)\Sigma_g^n\ar[d,"C",leftarrow]
\\
(fg)_\sharp\ar[r,"\mathfrak{p}_{fg}^n"]&
(fg)_!\Sigma_{fg}^n\ar[r,"\simeq"]&
f_!g_!\Sigma_{fg}^n
\end{tikzcd}
\end{equation}
commutes, where $c\colon Y\to X\times_S Y$ is the graph morphism, and $r_2\colon X\times_S Y\to Y$ is the projection.
\end{prop}
\begin{proof}
It suffices to check that the diagram
\[
\begin{tikzcd}[row sep=small, column sep=small]
f_\sharp g_\sharp\ar[dd,"\simeq"']\ar[rr,"\mathfrak{p}_f \mathfrak{p}_g"]\ar[rrdd,"(a)" description,phantom]&
&
f_!\Sigma_f g_!\Sigma_g\ar[d,"Ex"]\ar[r,"T"]\ar[d,"Ex"]\ar[rd,"(b)" description,phantom]&
f_!\Sigma_f^n g_!\Sigma_g^n\ar[d,"Ex"]
\\
&
&
f_!g_!\Thom(r_2,c)\Sigma_g\ar[d,"C",leftarrow]\ar[r,"T"]\ar[rd,"(c)",phantom,description,near start]&
f_!g_!\Thom_n(r_2,c)\Sigma_g^n\ar[d,"C",leftarrow]
\\
(fg)_\sharp\ar[r,"\mathfrak{p}_{fg}"]&
(fg)_!\Sigma_{fg}\ar[r,"\simeq"]&
f_!g_!\Sigma_{fg}\ar[r,"T"]&
f_!g_!\Sigma_{fg}^n
\end{tikzcd}
\]
commutes.
The square $(b)$ commutes by Proposition \ref{exchange.9}, and the square $(c)$ commutes by Proposition \ref{comp.3}.
For the diagram surrounding $(a)$, consider the induced commutative diagram with cartesian squares
\[
\begin{tikzcd}[column sep=small, row sep=small]
Y\ar[rd,"e"]
\\
&
Y\times_X Y\ar[dd,"t_1"']\ar[rr,"t_2"]\ar[rd,"d"]&
&
Y\ar[dd,"g",near start]\ar[rd,"c"]
\\
&
&
Y\times_S Y\ar[rr,"s_2",near start,crossing over]&
&
X\times_S Y\ar[dd,"r_1"]\ar[rr,"r_2"]&
&
Y\ar[dd,"g"]
\\
&
Y\ar[rr,"g",near start]\ar[rd,"b"']&
&
X\ar[rd,"a"]
\\
&
&
Y\times_S X\ar[dd,"q_1"']\ar[uu,"s_1"',leftarrow,crossing over,near end]\ar[rr,"q_2"]&
&
X\times_S X\ar[dd,"p_1"]\ar[rr,"p_2"]&
&
X\ar[dd,"f"]
\\
\\
&
&
Y\ar[rr,"g"]&
&
X\ar[rr,"f"]&
&
S,
\end{tikzcd}
\]
where $p_1$ and $p_2$ are the projections, and $a$ and $e$ are the diagonal morphisms.
To conclude, observe that the diagram
\[
\begin{tikzcd}[column sep=tiny, row sep=small]
f_\sharp g_\sharp\ar[rd,"\simeq"]\ar[rrrd,"\simeq",bend left=10,near end]\ar[rdd,bend right=20,"\simeq"']
\\
&
f_\sharp g_\sharp q_{1!}b_*t_{1!}e_*\ar[d,"\simeq"']\ar[r,"Ex"]&
f_\sharp p_{1!}q_{2\sharp}b_*t_{1!}e_*\ar[r,"Ex"]\ar[d,"\simeq"']&
f_\sharp p_{1!}a_*g_\sharp t_{1!}e_*\ar[d,"Ex"]
\\
&
f_\sharp g_\sharp q_{1!}s_{1!}d_*e_*\ar[r,"Ex"]&
f_\sharp p_{1*}q_{2\sharp}s_{1!}d_*e_*\ar[d,"Ex"']&
f_\sharp p_{1*}a_*g_!t_{2\sharp}e_*\ar[d,"\simeq"]
\\
&
&
f_\sharp p_{1!}r_{1!}s_{2\sharp}d_*e_*\ar[r,"Ex"]\ar[d,"Ex"']&
f_\sharp p_{1!}r_{1!}c_*t_{2\sharp}e_*\ar[r,"Ex"]&
f_!p_{2\sharp}r_{1!}c_*t_{2\sharp}e_*\ar[d,"Ex"]
\\
&
&
f_!p_{2\sharp}r_{1!}s_{2\sharp}d_*e_*\ar[rru,"Ex",bend right=10]\ar[d,"Ex"']&
&
f_!g_!r_{2\sharp}c_*t_{2\sharp}e_*
\\
&
&
f_!g_!r_{2\sharp}s_{2\sharp}d_*e_*\ar[rru,"Ex",bend right=10]
\end{tikzcd}
\]
commutes.
\end{proof}

\begin{rmk}
\label{functorial.4}
In \eqref{functorial.3.1}, the right vertical arrows are isomorphisms by Propositions \ref{exchange.10} and \ref{comp.4}.
\end{rmk}

\begin{prop}
\label{functorial.5}
Let $Y\xrightarrow{g} X\xrightarrow{f} S$ be separated exact log smooth morphisms in $\lSch/B$.
If $f$ and $g$ are pure,
then $fg$ is pure.
\end{prop}
\begin{proof}
This is an immediate consequence of Proposition \ref{functorial.3} and Remark \ref{functorial.4}.
\end{proof}

\begin{prop}
\label{functorial.6}
Let $f\colon X\to S$ be a separated exact log smooth morphism,
and let $\{g_i\colon U_i\to X\}_{i\in I}$ be a Zariski covering family with finite $I$.
If $fg_i$ is pure for every $i\in I$,
then $f$ is pure.
\end{prop}
\begin{proof}
We set $n:=\lvert I \rvert$.
The claim is trivial if $n=1$.
We proceed by induction on $n$.
Assume that $n>0$ and the claim holds for $n-1$.
Choose any $a\in I$,
and we set $U:=U_a$ and $V:=\cup_{i\in I-\{a\}} U_i$.
Let $u\colon U\to X$, $v\colon V\to X$, and $w\colon U\cap V\to X$ be the induced open immersions.
By assumption,
$fu$ is pure.
Use Proposition \ref{functorial.5} to see that $fw$ is pure.
By induction,
$fv$ is pure.

Together with Proposition \ref{functorial.3} and Remark \ref{functorial.4},
we see that the natural transformations
\begin{gather*}
f_\sharp u_\sharp
\xrightarrow{\mathfrak{p}_f^n}
f_!\Sigma_f^n u_\sharp,
\;
f_\sharp v_\sharp
\xrightarrow{\mathfrak{p}_f^n}
f_!\Sigma_f^n v_\sharp,
\;
f_\sharp w_\sharp
\xrightarrow{\mathfrak{p}_f^n}
f_!\Sigma_f^n w_\sharp
\end{gather*}
are isomorphisms.
To conclude,
observe that the induced commutative square
\[
\begin{tikzcd}
w_\sharp w^*\ar[d,"ad'"']\ar[r,"ad'"]&
v_\sharp v^*\ar[d,"ad'"]
\\
u_\sharp u^*\ar[r,"ad'"]&
\id
\end{tikzcd}
\]
is cocartesian by
\cite[Proposition 2.4.5]{logshriek}.
\end{proof}

\section{Poincar\'e duality}
\label{poincare}

Throughout this section,
we fix a log motivic $\infty$-category $\sT$
such that $M_S(X)$ is compact in $\sT^{ex}(X)$ for every vertical exact log smooth morphism $X\to S$,
or that $\sT^{ex}$ satisfies strict \'etale descent.

In \S \ref{affine},
we explain some basic properties of affine toric log schemes.
In \S \ref{virtual},
under some compactness assumption,
we show that for a virtual isomorphism $f$ in the sense of Definition \ref{homeomorphism.1},
$f^*$ is an equivalence of $\infty$-categories.
Using this,
we show in \S \ref{base2} a base change property different from \cite[Theorem 3.7.5]{logshriek}.

With these background materials in hand,
we show Poincar\'e duality in \S \ref{AQAP}.
We show Lefschetz duality in \S \ref{Lefschetz} for every log smooth morphism whose target has the trivial log structure.

\subsection{Affine toric log schemes}
\label{affine}

Let us review the structure $\A_P$ for an fs monoid $P$.
For an ideal $I$ of $P$,
we set
\[
\A_{P,I}:=\A_P\times_{\ul{\A_P}}\Spec(\Z[P]/\Z[I]).
\]
Here, $\Z[I]$ is an ideal of $\Z[P]$.
If $P$ is sharp,
we also set $\pt_P:=\A_{P,P^+}$,
where $P^+$ is the ideal of non-unit elements of $P$.

\begin{prop}
\label{affine.1}
Let $P$ be an fs monoid.
Then $\A_{P,P^+}\simeq \pt_{\ol{P}}\times \A_{P^*}$.
\end{prop}
\begin{proof}
By \cite[Lemma I.6.7]{zbMATH07027475}, there exists an isomorphism $P\simeq \ol{P}\oplus P^*$.
Since $\Z[P]/\Z[P^+]\simeq \Z[P^*]$, we have
\[
\A_{P,P^+}
\simeq
\Spec(\ol{P}\oplus P^*\to \Z[P^*]).
\]
This gives the desired isomorphism.
\end{proof}

\begin{prop}
\label{affine.3}
Let $P$ be an fs monoid, and let $x$ be a point of $\A_P$.
For any face $F$ of $P$, $x\in \A_{P_F}$ if and only if the kernel of $\ol{P}\to \ol{\cM}_{\A_P,x}$ contains $\ol{F}$.
\end{prop}
\begin{proof}
The only if direction is trivial.

For the if direction, assume that $x$ corresponds to a homomorphism $\Z[P]\to k$ with a field $k$.
If the kernel of $\ol{P}\to \ol{\cM}_{\A_P,x}$ contains $\ol{F}$, then the composition $P\to \Z[P]\to k$ factors through $P_F$.
This means $x\in \A_{P_F}$.
\end{proof}

For a ring $A$ and an ideal $J$, we have the relation
\[
\Spec(A)-\Spec(A/J)
=
\bigcup_{f\in J}\Spec(A_f).
\]
This implies that for a monoid $P$ and an ideal $I$, we have the relation
\begin{equation}
\label{affine.2.1}
\A_P-\A_{P,I}
=
\bigcup_{f\in I}\A_{P_f}.
\end{equation}

\begin{prop}
\label{affine.2}
Let $P$ be an fs monoid, and let $x$ be a point of $\A_P$.
Then $x\in \A_{P,P^+}$ if and only if $\ol{P}\to \ol{\cM}_{\A_P,x}$ is an isomorphism.
\end{prop}
\begin{proof}
The only if direction follows from Proposition \ref{affine.1}.

Use \eqref{affine.2.1} for $I=P^+$ to see that if $x\in \A_P-\A_{P,P^+}$, then $x\in \A_{P_{f}}$ for some $f\in P^+$.
In this case, $\ol{P}\to \ol{\cM}_{\A_P,x}$ is not an isomorphism by Proposition \ref{affine.3}.
\end{proof}

\begin{prop}
\label{affine.4}
Let $P$ be an fs monoid, and let $x$ be a point of $\A_P$.
Then for every face $F$ of $P$,
$x\in \A_{P,P-F}\cap \A_{P_F}=\A_{P_F,P_F^+}$ if and only if the kernel of $\ol{P}\to \ol{\cM}_{\A_P,x}$ is $\ol{F}$.
Hence we have the stratification
\[
\coprod_{F\cap I=\emptyset}
\A_{P_F,P_F^+}
\]
on $\A_{P,I}$ for every ideal $I$ of $P$,
where the disjoint union runs over the faces $F$ of $P$ such that $F\cap I=\emptyset$.
\end{prop}
\begin{proof}
Use Propositions \ref{affine.3} and \ref{affine.2} for the first claim.
This yields the desired stratification.
\end{proof}

\begin{const}
\label{affine.5}
Let $f\colon X\to S$ be a morphism of fs log schemes with a chart $P\to Q$.
Assume that the underlying morphism of schemes $\ul{f}\colon \ul{X}\to \ul{S}$ is an isomorphism.
By Proposition \ref{affine.4},
we have the stratification
\[
\coprod_{G\subset Q} X\times_{\A_Q}  \A_{Q_G,Q_G^+}
\]
on $X$,
where $G$ runs over the faces of $Q$.
Furthermore,
$f$ sends $X\times_{\A_Q} \A_{Q_G,Q_G^+}$ into $S\times_{\A_P}\A_{P_F,P_F^+}$ with $F:=\theta^{-1}(G)$.
Together with Proposition \ref{affine.1}.
we have a stratification $\coprod_{i\in I} \ul{V_i}$ on $\ul{S}\simeq \ul{X}$ with finite $I$ such that for every $i\in I$,
the induced morphism $X\times_{\ul{S}}\ul{V_i}\to S\times_{\ul{S}}\ul{V_i}$ is of the form $V_i\times \pt_{Q/G}\to V_i\times \pt_{P/\theta^{-1}(G)}$ for some face $G$ of $Q$.
\end{const}

\subsection{Motivic invariance under virtual isomorphisms}
\label{virtual}

The main result in this section is divided into Theorems \ref{homeomorphism.5} and \ref{homeomorphism.9}.

\begin{df}
\label{homeomorphism.1}
A morphism of saturated log schemes $f\colon X\to S$ is a \emph{virtual isomorphism} if $\ul{f}\colon \ul{X}\to \ul{S}$ is an isomorphism and the induced morphism of sheaves $\cM_S^\gp\to \cM_X^\gp$ is an isomorphism.
\end{df}

\begin{rmk}
\label{homeomorphism.3}
A morphism of fs log schemes is a virtual isomorphism if and only if its virtualization in the sense of \cite[\S 2]{2209.03720} is an isomorphism.

According to \cite[\S 1.2]{zbMATH01367783},
the Kato-Nakayama realization $X^{\log}$ of an fs log scheme $X$ over $\C$ only depends on $X$ and $\cM_X^\gp$.
Hence a virtual isomorphism of fs log schemes over $\C$ becomes a homeomorphism after applying the Kato-Nakayama realization functor.
\end{rmk}

\begin{lem}
\label{homeomorphism.2}
Let $\theta\colon P\to Q$ be an injective homomorphism of integral monoids such that $\theta^\gp$ is an isomorphism.
Then the cokernel of $\theta$ is trivial.
\end{lem}
\begin{proof}
For all $q\in Q$, there exist $p,p'\in P$ such that $q+p=p'$ since $\theta^\gp$ is an isomorphism.
The description of $Q/P$ in \cite[Example I.1.1.6]{Ogu} shows that $Q/P$ is trivial.
\end{proof}

\begin{lem}
\label{homeomorphism.6}
Let
\[
\begin{tikzcd}
P\ar[d,"\theta"']\ar[r,"\eta"]&
P'\ar[d,"\theta'"]
\\
Q\ar[r,"\eta'"]&
Q'
\end{tikzcd}
\]
be a cocartesian square of sharp fs monoids.
Assume the following conditions:
\begin{enumerate}
\item[\textup{(i)}] $\theta'^\gp$ is an isomorphism.
\item[\textup{(ii)}] If $F$ is a face of $P'$ such that $\eta^{-1}(F)=0$, then $\theta'(F)$ is a face of $Q'$.
\item[\textup{(iii)}] If $G$ is a face of $Q'$ such that $\eta'^{-1}(G)=0$, then $G=\theta'(F)$ for some face $F$ of $P'$.
\end{enumerate}
Then the induced morphism
\[
f\colon
\A_{Q',J'}
\to
\A_{P',I'}
\]
is a virtual isomorphism,
where $I'$ and $J'$ are the ideals of $P'$ and $Q'$ generated by $\eta(P^+)$ and $\eta'(Q^+)$.
\end{lem}
\begin{proof}
By Proposition \ref{affine.4},
we have the stratifications
\[
\coprod_{F\cap I'=\emptyset} \A_{P_F',P_F'^+}
\text{ and }
\coprod_{G\cap J'=\emptyset} \A_{Q_G',Q_G'^+}
\]
on $\A_{P',I'}$ and $\A_{Q',J'}$,
where the coproducts run over the faces $F$ and $G$ of $P'$ and $Q'$ such that $F\cap I'=\emptyset$ and $G\cap J'=\emptyset$.
The condition $F\cap I'=\emptyset$ (resp.\ $G\cap J'=\emptyset$) is equivalent to the condition $\eta^{-1}(F)=0$ (resp.\ $\eta'^{-1}(G)=0$).
Hence the conditions (ii) and (iii) imply that we have the restriction
\[
f_F
\colon
\A_{Q_{\theta'(F)}',Q_{\theta'(F)}'^+}
\to
\A_{P_F',P_F'^+}
\]
of $f$ for every face $F$ of $P'$ such that $\eta^{-1}(F)=0$ and
\[
\coprod_{\eta^{-1}(F)=0}
\A_{Q_{\theta'(F)}',Q_{\theta'(F)}'^+}
\]
is a stratification on $\A_{Q',J'}$.
Due to Proposition \ref{affine.1},
$f_F$ can be written as the induced morphism
\[
\pt_{Q'/\theta'(F)}\times \A_{\theta'(F)^\gp}
\to
\pt_{P/F}\times \A_{F^\gp}.
\]
The condition (i) imply that $F\to \theta(F)$ and $(P/F)^\gp\to (Q'/\theta'(F))^\gp$ are isomorphisms.
Hence $f_F$ is a virtual isomorphism.
Together with \cite[Corollaire IV.18.12.6]{EGAIVIV},
we deduce that $f$ is a virtual isomorphism.
\end{proof}

\begin{prop}
\label{homeomorphism.4}
Let $f\colon X\to S$ be a virtual isomorphism in $\lSch/B$.
Then $f^*$ is fully faithful.
\end{prop}
\begin{proof}
Equivalently,
we need to show that the unit $\id \xrightarrow{ad} f_*f^*$ is an isomorphism.
We proceed by induction on $r:=\max_{s\in S}\ol{\cM}_{S,s}^\gp$.
If $r\leq 1$, then $f$ is an isomorphism.
Hence we assume $r>1$.

(I) \emph{Locality of the question.}
The question is Zariski local on $S$, so we may assume that $S$ has a neat chart $P$.
Lemma \ref{homeomorphism.2} shows $\cM_{X/S}^\gp\simeq 0$,
see \cite[Definition A.2]{divspc} for the notation $\cM_{X/S}$.
Since $\underline{f}$ is an isomorphism, the question is also Zariski local on $X$.
Hence by \cite[Theorem II.2.4.4]{Ogu}, we may assume that $f$ admits a neat chart $\theta\colon P\to Q$.
According to \cite[Remark II.2.4.5]{Ogu}, $Q$ is a neat chart of $X$.

Let $i\colon Z\to S$ be a strict closed immersion, and let $j\colon U\to S$ be its open complement.
By (Loc),
it suffices to show that the natural transformations
\[
i^*\xrightarrow{ad} i^*f_*f^*
\text{ and }
j^*\xrightarrow{ad} j^*f_*f^*
\]
are isomorphisms.
Together with
($\eSm$-BC), (Supp), and
\cite[Proposition 3.1.1(2)]{logshriek},
we reduce to the question for $X\times_S Z\to Z$ and $X\times_S U\to U$.
Hence by induction,
we reduce to the case when the chart $S\to \A_P$ factors through $\pt_P$.
In this case, for every point $x$ of $X$, we have $\ol{\cM}_{X,x}^\gp\simeq P^\gp$.
This means that the chart $X\to \A_Q$ factors through $\pt_Q$.
The induced morphism $S\to \ul{S}\times \pt_P$ is strict, and its underlying morphism is an isomorphism.
Hence we have $S\simeq \ul{S}\times \pt_P$.
Similarly, we have $X\simeq \ul{X}\times \pt_Q$.
In summary, we reduce to the case when $f$ is the morphism
\[
\id\times \pt_\theta\colon \ul{S}\times \pt_Q\to \ul{S}\times \pt_P.
\]

(II) \emph{Reduction of $Q$.} Choose generators $a_1,\ldots,a_m$ of $Q$.
Consider the inclusions of sharp fs monoids
\[
P\to (P+\langle a_1\rangle)^\sat\to \cdots \to (P+\langle a_1,\ldots,a_m\rangle)^\sat.\]
If we show the question for each inclusion, then we are done.
Hence we reduce to the case when $Q=(P+\langle a\rangle )^\sat$ for some $a\in P^\gp$.
  
(III) \emph{Construction of fans.} We set $C:=P_\Q$ and $D=Q_\Q$.
Consider their dual cones $C^\vee$ and $D^\vee$.
Observe that we have $C\subset D$ and $D^\vee\subset C^\vee$.
Choose any point $v$ in the interior of $D^\vee$.
There exists a subdivision of $D^\vee$ consisting of simplicial cones.
Extend this subdivision to a subdivision of $C^\vee$ consisting of simplicial cones.

Let $\{\sigma_i\}_{i\in I}$ be the set of $(n-1)$-dimensional cones $\sigma$ of the subdivision of $C^\vee$ such that $\sigma$ is contained in a face of $C^\vee$.
We set
\[
C_i^\vee:=\sigma_i+\langle v\rangle,
\text{ and }
D_i^\vee:=C_i^\vee\cap D^\vee.
\]
Let $C_i$ and $D_i$ be their dual cones.
We set
\[
P_i:=C_i\cap P^\gp
\text{ and }
Q_i:=D_i\cap Q^\gp.
\]
We have $C_i=D_i$ or $C_i\neq D_i$.
To help the reader,
we include Figure \ref{fig2} illustrating this situation.

\begin{figure}
\centering
\begin{tikzpicture}
\draw (1,0)--(2,2)--(0,2.5)--(-2,2)--(-1,0)--(1,0);
\draw[dashed] (0,1)--(1,0);
\draw[dashed] (0,1)--(2,2);
\draw[dashed] (0,1)--(0,2.5);
\draw[dashed] (0,1)--(-2,2);
\draw[dashed] (0,1)--(-1,0);
\fill[fill=black] (0,1) circle (2pt);
\node at (0,0.25) {$C_1^\vee$};
\node at (1,0.8) {$C_2^\vee$};
\node at (0.6,1.8){$C_3^\vee$};
\node at (-0.6,1.8){$C_4^\vee$};
\node at (-1,0.8) {$C_5^\vee$};
\node at (0,0.7) {$v$};
\node at (0,-0.3) {$\sigma_1$};
\node at (1.8,1) {$\sigma_2$};
\node at (1.4,2.35) {$\sigma_3$};
\node at (-1.4,2.35) {$\sigma_4$};
\node at (-1.8,1) {$\sigma_5$};
\begin{scope}[shift={(6,0)}]
\draw (1,0)--(2,2)--(-2,2)--(-1,0)--(1,0);
\draw[dashed] (0,1)--(1,0);
\draw[dashed] (0,1)--(2,2);
\draw[dashed] (0,1)--(0,2);
\draw[dashed] (0,1)--(-2,2);
\draw[dashed] (0,1)--(-1,0);
\fill[fill=black] (0,1) circle (2pt);
\node at (0,0.25) {$D_1^\vee$};
\node at (1,0.8) {$D_2^\vee$};
\node at (0.4,1.6){$D_3^\vee$};
\node at (-0.4,1.6){$D_4^\vee$};
\node at (-1,0.8) {$D_5^\vee$};
\node at (0,0.7) {$v$};
\node at (0,2.3) {$\langle a \rangle^\bot$};
\node at (0,-0.3) {$\sigma_1$};
\node at (1.8,1) {$\sigma_2$};
\node at (-1.8,1) {$\sigma_5$};
\end{scope}
\end{tikzpicture}
\caption{The illustrations of $C_i^\vee$ and $D_i^\vee$}\label{fig2}
\end{figure}

Assume $C_i\neq D_i$.
Let $F_1,\ldots,F_{n-1},\sigma_i$ be the $(n-1)$-dimensional faces of $C_i^\vee$.
Among these, the faces other than $\sigma_i$ contain $v$.
This means
\[
C_i=\langle b_1,\ldots,b_{n-1},r_i\rangle
\]
for some $b_1,\ldots,b_{n-1}\in \langle v\rangle^\bot$ and $r_i\in \sigma_i^\bot$.
On the other hand,
\[
F_1\cap D_i^\vee,\ldots,F_{n-1}\cap D_i,\langle a \rangle^\bot \cap D_i^\vee
\]
are the $(n-1)$-dimensional faces of $D_i^\vee$.
This means
\[
D_i=\langle b_1,\ldots,b_{n-1},a \rangle.
\]
Since $v$ is in the interiors of $C^\vee$ and $D^\vee$, we have
\[
\langle v\rangle^\bot \cap C
=
\langle v \rangle^\bot \cap D
=
0.
\]
Hence if $F$ and $G$ are faces of $C_i$ and $D_i$, then we have
\begin{gather}
\label{homeomorphism.4.2}
F\cap C=\langle 0\rangle \;\;\Leftrightarrow \;\;F\subset \langle b_1,\ldots,b_{n-1}\rangle,
\\
G\cap D=\langle 0\rangle \;\;\Leftrightarrow \;\;G\subset \langle b_1,\ldots,b_{n-1}\rangle.
\end{gather}
It follows that the cocartesian square
\[
\begin{tikzcd}
P\ar[r]\ar[d]&
P_i\ar[d]
\\
Q\ar[r]&
Q_i
\end{tikzcd}
\]
satisfies the condition of Lemma \ref{homeomorphism.6},
so the induced morphism
\begin{equation}
\label{homeomorphism.4.1}
\ul{S}\times \A_{Q_i}\times_{\A_Q}\pt_Q
\to
\ul{S}\times \A_{P_i}\times_{\A_P}\pt_P
\end{equation}
is a virtual isomorphism.

If $C_i=D_i$, then we can similarly show that \eqref{homeomorphism.4.1} is a virtual isomorphism using \eqref{homeomorphism.4.2}.

Let $\Delta$ (resp.\ $\Sigma$) be the fans whose set of maximal cones corresponds to $\{C_i^\vee\}_{i\in I}$ (resp.\ $\{D_i^\vee\}_{i\in I}$).
We have the induced commutative square
\[
\begin{tikzcd}
\ul{S}\times \T_\Sigma\times_{\A_Q}\pt_Q\ar[d,"g'"']\ar[r,"f'"]&
\ul{S}\times \T_\Delta\times_{\A_P}\pt_P\ar[d,"g"]
\\
\ul{S}\times \pt_Q\ar[r,"f"]&
\ul{S}\times \pt_P.
\end{tikzcd}
\]
We have shown that $f'$ is a virtual isomorphism.

Consider the induced commutative diagram
\[\begin{tikzcd}
\id \ar[rr,"ad"]\ar[d,"ad"']&
&
f_*f^*\ar[d,"ad"]
\\
g_*g^*\ar[r,"ad"]&
g_*f_*'f'^*g^*\ar[r,"\simeq"]&
f_*g_*'g'^*f^*.
\end{tikzcd}
\]
The vertical arrows are isomorphisms since $g$ and $g'$ are dividing covers.
Hence the question for $f$ reduces to the question for $f'$.

By Zariski descent and induction on $r$, we reduce to the questions for $(P,Q)=(P_i,Q_i)$ for $i\in I$.
In particular, we may assume
\[
(P_i)_\Q
=
\langle b_1,\ldots,b_{r-1},b_r\rangle
\text{ and }
(Q_i)_\Q
=
\langle b_1,\ldots,b_{r-1},a\rangle
\]
for some $a,b_1,\ldots,b_r\in P_i$.

(IV) \emph{Final step of the proof.}
We set
\[
F:=\langle b_1\rangle\cap P,
\;
G:=\langle b_1\rangle \cap  Q,
\;
P':=\langle b_2,\ldots,b_r\rangle \cap P,
\;
Q'=\langle b_2,\ldots,b_{r-1},a\rangle \cap Q.
\]
Consider the induced commutative diagram with cartesian squares
\[
\begin{tikzcd}
\ul{S}\times \pt_Q\ar[r,"i'"]\ar[d,"f"']&
\ul{S}\times\A_{(Q,Q-G)}\ar[d,"f'"']\ar[rr,"j'",leftarrow]\arrow[rdd,"q",bend left,near start]&
&
\ul{S}\times(\A_{(Q,Q-G)}-\pt_Q)\arrow[d,"f''"]
\\
\ul{S}\times \pt_P\ar[r,"i"]&
\ul{S}\times\A_{(P,P-F)}\ar[rr,"j",leftarrow,crossing over]\ar[d,"p"']&
&
\ul{S}\times(\A_{(P,P-F)}-\pt_P)
\\
&
\ul{S}\times\A_{P'}\ar[r,"g",leftarrow]&
\ul{S}\times\A_{Q'}.
\end{tikzcd}
\]
We have
$
\pt_P
=
\partial_{\A_{P'}}\A_{(P,P-F)}
\text{ and }
\pt_Q
=
\partial_{\A_{P'}}\A_{(Q,Q-G)}.
$
Hence by $\ver$-invariance, the natural transformations
\[
p^*\xrightarrow{ad} j_*j^*p^*
\text{ and }
q^*\xrightarrow{ad} j_*'j'^*q^*
\]
are isomorphisms.
It follows that in the commutative diagram
\[
\begin{tikzcd}
p^*\ar[rr,"ad"]\ar[d,"ad"',"\simeq"]&
&
f_*'f'^*p^*\ar[d,"\simeq"]
\\    j_*j^*p^*\ar[r,"ad","\simeq"']&
j_*f_*''f''^*j^*p^*\ar[r,"\simeq"]&
p_*'j_*'j'^*f'^*p^*,
\end{tikzcd}
\]
the upper horizontal arrow is an isomorphism.

We have the commutative diagram
\[
\begin{tikzcd}    i^*p^*\ar[r,"ad","\simeq"']\arrow[rrd,"ad"',bend right=10]&
i^*f_*'f'^*p^*\ar[r,"Ex"]&
f_*i'^*f'^*p^*\ar[d,"\simeq"]
\\
&
&
f_*f^*i^*p^*.
\end{tikzcd}
\]
The upper right horizontal arrow is an isomorphism by (Supp) and 
\cite[Proposition 3.1.1(2)]{logshriek}.
Hence the diagonal arrow is an isomorphism.
In particular, the morphism
\[
\unit\xrightarrow{ad} f_*f^*\unit\]
is an isomorphism.

For every object $\cF$ of $\sT^{ex}(S)$, we have the commutative diagram
\[
\begin{tikzcd}
\cF\ar[r,"\simeq"]\arrow[rrd,"ad"',bend right=10]&
\unit\otimes \cF\ar[r,"ad","\simeq"']&
f_*f^*\unit \otimes \cF\ar[d,"Ex"]
\\
&
&
f_*f^*\cF.
\end{tikzcd}
\]
By (PF), the right vertical arrow is an isomorphism.
Hence the diagonal arrow is also an isomorphism, which completes the proof.
\end{proof}

\begin{thm}
\label{homeomorphism.5}
Assume that $M_S(X)$ is compact in $\sT^{ex}(X)$ for every vertical exact log smooth morphism $X\to S$.
Let $f\colon X\to S$ be a virtual isomorphism in $\lSch/B$.
Then $f^*$ is an equivalence of $\infty$-categories.
\end{thm}
\begin{proof}
By Proposition \ref{homeomorphism.4}, it remains to show that either the counit $f^*f_*\xrightarrow{ad'} \id$ is an isomorphism.

(I) \emph{Locality of the question.}
The question is Zariski local on $\ul{X}\simeq \ul{S}$.
Let $i\colon Z\to S$ be a strict closed immersion, and let $j\colon U\to S$ be the open complement.
By (Loc), it suffices to show that the natural transformations
\[
i^*f^*f_*\xrightarrow{ad'} i^*
\text{ and }
j^*f^*f_*\xrightarrow{ad'} j^*
\]
are isomorphisms.
Let $g\colon Z\times_S X\to Z$ and $h\colon U\times_S X\to U$ be the projections.
By ($\eSm$-BC), (Supp), and
\cite[Proposition 3.1.1(2)]{logshriek},
it is equivalent to showing that the natural transformations
\[
g^*g_*i^*\xrightarrow{ad'} i^*
\text{ and }
h^*h_*j^*\xrightarrow{ad'} j^*
\]
are isomorphisms.
Use this reduction method and argue as in the proof of Proposition \ref{homeomorphism.4} to reduce to the case when $f$ is equal to
\[
\id \times \pt_\theta \colon \ul{S}\times \pt_Q\to \ul{S}\times \pt_P
\]
for some homomorphism of sharp fs monoids $\theta\colon P\to Q$.

(II) \emph{Reduction using a conservative functor.} We only need to show that the morphism
\[
f^*f_*M(U)\xrightarrow{ad'} M(U)
\]
is an isomorphism for all $U\in \eSm/X$ since $f^*$ and $f_*$ preserve colimits.

By \cite[Proposition 2.4.3]{logshriek},
there exists a Kummer log smooth morphism $g'\colon X'\to X$ such that $g'^*$ is conservative and the projection $U\times_X X'\to X'$ is saturated.
The explicit description in the proof of \cite[Proposition 3.3.3]{logGysin} tells that $X':=X\times_{\A_Q}\A_{Q'}$ and $g'$ is the projection, where $Q':=Q^\gp\oplus Q$, and the homomorphism $Q\to Q'$ in the formulation of $X'$ is given by $x\mapsto (x,nx)$ for some integer $n\geq 1$.

Consider the homomorphism $P\to P':=P^\gp\oplus P$ given by $x\mapsto (x,nx)$.
The induced square
\[
C
:=
\begin{tikzcd}
P\ar[d,"\theta"']\ar[r]&
P'\ar[d]
\\
Q\ar[r]&
Q'
\end{tikzcd}
\]
is cocartesian since $C^\gp$ is cocartesian and $(nx,ny)$ is in the image of $Q\to Q'$ for every $(x,y)\in Q'$.
This means that we have a cartesian square
\[
\begin{tikzcd}
X'\ar[d,"f'"']\ar[r,"g'"]&
X\ar[d,"f"]
\\
S'\ar[r,"g"]&
S,
\end{tikzcd}
\]
where $S':=S\times_{\A_P}\A_{P'}$ and $g$ is the projection.
Observe that $g$ is Kummer log smooth by \cite[Theorem IV.3.1.8]{Ogu} since $P^\gp$ is torsion free by \cite[Proposition I.1.3.5(2)]{Ogu}.

We only need to show that the morphism
\[
g'^*f^*f_*M(U)
\xrightarrow{ad'}
g'^*M(U)
\]
is an isomorphism.
By ($\eSm$-BC),
it suffices to show that the morphism
\[
f'^*f_*'M(U\times_X X')
\xrightarrow{ad'}
M(U\times_X X')
\]
is an isomorphism.

(III) \emph{Further reduction.}
By \cite[Propositions 2.5.3, 2.5.4]{logshriek},
it suffices to show that for all strict proper morphism $w\colon W\to X'$, the morphism
\[
f'^*f_*'w_*\unit
\xrightarrow{ad'}
w_*\unit
\]
is an isomorphism.
We set $V:=S'\times_{\ul{S'}}\ul{W}$ to have a cartesian square
\[
\begin{tikzcd}
W\ar[r,"w"]\ar[d,"f''"']&
X'\ar[d,"f'"]
\\
V\ar[r,"v"]&
S'.
\end{tikzcd}
\]
By
\cite[Proposition 2.4.13]{logshriek},
we reduce to showing that the morphism
\[
f''^*f_*''\unit
\xrightarrow{ad'}
\unit
\]
is an isomorphism.

(IV) \emph{Final step of the proof.}
Observe that $f''$ is a virtual isomorphism.
Since $\unit\simeq f''^*\unit$, it suffices to show that the second arrow in 
\[
f''^*\unit \xrightarrow{ad}
f''^*f_*''f''^*\unit
\xrightarrow{ad'}
f''^*\unit
\]
is an isomorphism.
Since the composition is an isomorphism, it suffices to show that the first arrow is an isomorphism.
This is a consequence of Proposition \ref{homeomorphism.4}.
\end{proof}

\begin{lem}
\label{homeomorphism.10}
Let $\theta\colon P\to Q$ be a homomorphism of sharp fs monoids such that $\theta^\gp$ is an isomorphism,
and let $\eta'\colon Q\to Q'$ be a homomorphism of fs monoids such that $\ol{\eta'}$ is Kummer.
Then there exists a cocartesian square of fs monoids
\begin{equation}
\label{homeomorphism.10.1}
\begin{tikzcd}
P\ar[r,"\theta"]\ar[d,"\eta"']&
Q\ar[d,"\eta'"]
\\
P'\ar[r]&
Q'
\end{tikzcd}
\end{equation}
such that $\ol{\eta}$ is Kummer,
$P'^*\xrightarrow{\simeq} Q'^*$,
and $P'^\gp \xrightarrow{\simeq} Q'^\gp$.
\end{lem}
\begin{proof}
Let $P'$ be the submonoid of $Q'$ consisting of elements $p'$ such that $np'\in P+ Q'^*$ for some $n>0$,
and let $\eta\colon P\to P'$ be the inclusion.
By construction, $\ol{\eta}$ is Kummer.
Also, we have $Q'^*\subset P'$,
which implies $P'^*\simeq Q'^*$.

Let $q'\in Q'$.
Then there exists $m>0$ such that $mq'=q+u$ for some $q\in Q$ and $u\in Q'^*$.
Let $r$ be the rank of $P^\gp$.
By toric resolution of singularities \cite[Theorem 11.1.9]{CLStoric},
there exist $p_1,\ldots,p_r\in P$ forming a basis of $P^\gp$.
We can write $q=a_1p_1+\cdots + a_rp_r$ with $a_1,\ldots,a_r\in \Z$ since $\theta^\gp$ is an isomorphism.

Choose integers $b_1,\ldots,b_r>0$ such that $a_1+mb_1,\ldots,a_r+mb_r>0$.
Let $p=mb_1p_1+\cdots + mb_rp_r\in P$.
Then $mp+q\in P$,
so $m(p+q')\in P+Q'^*$.
Hence $p+q'\in P'$, which implies $q'\in P'^\gp$.
This shows $P'^\gp\simeq Q'^\gp$.

We also have $mq'\in Q\oplus_P P'$ since $Q'^*\subset P'$.
It follows that $q'\in Q\oplus_P P'$.
This proves that \eqref{homeomorphism.10.1} is cocartesian.
\end{proof}

\begin{lem}
\label{homeomorphism.8}
Let $\eta\colon P\to P'$ be a Kummer homomorphism of sharp fs monoids.
Then
\[
(\ul{\pt_P \times_{\A_P} \A_{P'}})_{\red}
\simeq
\ul{\pt_P}.
\]
\end{lem}
\begin{proof}
We can write $\ul{\pt_P}\simeq \Spec(\Z[\{x^p\}_{p\in P}]/(x^p)_{p\in P^+})$ and
\begin{equation}
\label{homeomorphism.8.1}
\ul{\pt_P \times_{\A_P} \A_{P'}}
\simeq
\Spec(\Z[\{x^{p'}\}_{p'\in P'}]/(x^{\eta(p)})_{p\in P^+}).
\end{equation}
Given $p'\in P'^+$, there exists $m>0$ such that $mp'\in \eta(P)$ since $\eta$ is Kummer.
To conclude,
observe that $(x^{p'})^m=0$ in the ring used in \eqref{homeomorphism.8.1}.
\end{proof}

\begin{thm}
\label{homeomorphism.9}
Assume that $\sT^{ex}$ satisfies strict \'etale descent.
Let $f\colon X\to S$ be a virtual isomorphism in $\lSch/B$.
Then $f^*$ is an equivalence of $\infty$-categories.
\end{thm}
\begin{proof}
As in the proof of Theorem \ref{homeomorphism.5},
we reduce to the case where $f$ is equal to
\[
\id \times \pt_\theta \colon \ul{S}\times \pt_Q\to \ul{S} \times \pt_P
\]
for some homomorphism of sharp fs monoids $P\to Q$.
We already know that $f^*$ is fully faithful by Proposition \ref{homeomorphism.4},
so we only need to show that $f^*$ is essentially surjective.
Together with \cite[Proposition 2.6.4]{logshriek},
it suffices to show that for every Kummer smooth morphism $g'\colon Y'\to \ul{S}\times \pt_Q$,
strict \'etale locally on $Y'$,
there exists a cartesian square
\[
\begin{tikzcd}
Y'\ar[d,"g'"']\ar[r]&
Y\ar[d,"g"]
\\
\ul{S}\times \pt_Q\ar[r]&
\ul{S}\times \pt_P
\end{tikzcd}
\]
such that $g$ is Kummer smooth.
By \cite[Theorem IV.3.3.1]{Ogu},
we may assume that there exists a homomorphism $\eta\colon Q\to Q'$ such that $\lvert \cok(\eta^\gp)^\tor\rvert$ is invertible in $S$, $\ol{\eta}\colon Q\to \ol{Q'}$ is Kummer, and the induced morphism $h'\colon Y'\to \ul{S}\times \pt_Q \times_{\A_Q} \A_{Q'}$ is strict \'etale.

Lemma \ref{homeomorphism.10} yields a cocartesian square of the form \eqref{homeomorphism.10.1}.
By Lemma \ref{homeomorphism.8},
we have
\[
(\ul{\pt_Q \times_{\A_Q}\A_{Q'}})_\red
\simeq
\ul{\pt_Q}
\simeq
\ul{\pt_P}
\simeq
(\ul{\pt_P \times_{\A_P}\A_{P'}})_\red.
\]
It follows that the induced morphism
\[
\ul{S\times \pt_Q\times_{\A_Q} \A_{Q'}} \to 
\ul{S\times \pt_P \times_{\A_P} \A_{P'}}
\]
is a universal homeomorphism,
so there exists a cartesian square
\[
\begin{tikzcd}
Y'\ar[d,"h'"']\ar[r]&
Y\ar[d,"h"]
\\
\ul{S}\times \pt_Q\times_{\A_Q} \A_{Q'}\ar[r]&
\ul{S}\times \pt_P\times_{\A_P} \A_{P'}
\end{tikzcd}
\]
such that $h$ is strict \'etale.
To conclude,
observe that the composite $Y\to \ul{S}\times \pt_P$ is Kummer smooth.
\end{proof}

\subsection{Base change property}
\label{base2}
The purpose of this section is to prove a base change property that is different from \cite[Proposition 3.7.1]{logshriek}.

\begin{lem}
\label{base2.6}
Let
\[
\begin{tikzcd}
U'\ar[d,"u'"']\ar[r,"g''"]&
U\ar[d,"u"]
\\
X'\ar[d,"f'"']\ar[r,"g'"]&
X\ar[d,"f"]
\\
S'\ar[r,"g"]&
S
\end{tikzcd}
\]
be a diagram in $\lSch/B$ with cartesian squares.
If the natural transformations $g^*f_*\xrightarrow{Ex} f_*'g'^*$ and $g'^*u_*\xrightarrow{Ex} u_*'g''^*$ are isomorphisms,
then the natural transformation $g^*(fu)_*\xrightarrow{Ex} (f'u')_*g''^*$ is an isomorphism.
\end{lem}
\begin{proof}
The diagram
\[
\begin{tikzcd}
g^*(fu)_*\ar[rr,"Ex"]\ar[d,"\simeq"']&
&
(f'u')_*g''^*\ar[d,"\simeq"]
\\
g^*f_*u_*\ar[r,"Ex"]&
f_*'g'^*u_*\ar[r,"Ex"]&
f_*'u_*'g''^*
\end{tikzcd}
\]
commutes,
which shows the claim.
\end{proof}

\begin{lem}
\label{base2.1}
Let $g\colon S'\to S$ be a morphism in $\lSch/B$,
and let $\theta\colon \N\to \N^2$ be the diagonal homomorphism of fs monoids.
Consider the induced cartesian square
\[
\begin{tikzcd}
S'\times \pt_{\N^2}\ar[d,"f'"']\ar[r,"g'"]&
S\times \pt_{\N^2}\ar[d,"f"]
\\
S'\times \pt_\N\ar[r,"g"]&
S\times \pt_\N,
\end{tikzcd}
\]
where $f:=\id \times \pt_\theta$.
Then the natural transformation
\[
Ex\colon g^*f_*
\to
f_*'g'^*
\]
is an isomorphism.
\end{lem}
\begin{proof}
Let $W$ be the gluing of
\[
\Spec(\N x\oplus \N y\to \Z[x,y]),
\;
\Spec(\N(xy)\to \Z[xy,y^{-1}]),
\;
\Spec(\N(xy)\to \Z[xy,x^{-1}]).
\]
We have the morphism $W\to \A_\N=\Spec(\N t\to \Z[t])$ obtained by the assignment $t\mapsto xy$.
Consider the induced commutative diagram with cartesian squares
\[
\begin{tikzcd}
S'\times \pt_{\N^2}\ar[d,"i'"']\ar[r,"g'"]&
S\times \pt_{\N^2}\ar[d,"i"]
\\
S'\times W\times_{\A_\N}\pt_\N\ar[r,"g''"]\ar[d,"p'"']&
S\times W\times_{\A_\N}\pt_\N\ar[d,"p"]
\\
S'\times \pt_\N\ar[r,"g"]&
S\times \pt_\N.
\end{tikzcd}
\]
Since $i$ is a strict closed immersion,
\cite[Proposition 2.4.13]{logshriek} implies that $g''^*i_*\xrightarrow{Ex} i_*'g'^*$ is an isomorphism.
By Propositions \ref{puritytran.3} and \ref{functorial.7},
$g^*p_*\xrightarrow{Ex} p_*'g''^*$ is an isomorphism.
Combine these two with Lemma \ref{base2.6} to conclude.
\end{proof}

\begin{lem}
\label{base2.2}
Let $g\colon S'\to S$ be a morphism in $\lSch/B$,
and let $\theta\colon \N\to \N^2$ be the first inclusion.
Consider the induced cartesian square
\[
\begin{tikzcd}
S'\times \pt_{\N^2}\ar[d,"f'"']\ar[r,"g'"]&
S\times \pt_{\N^2}\ar[d,"f"]
\\
S'\times \pt_\N\ar[r,"g"]&
S\times \pt_\N,
\end{tikzcd}
\]
where $f:=\id \times \pt_\theta$.
Then the natural transformation
\[
Ex\colon g^*f_*
\to
f_*'g'^*
\]
is an isomorphism.
\end{lem}
\begin{proof}
Consider the homomorphism $\theta\colon \N^2\to \N^2$ given by $\theta(x,y):=(x,x+y)$ for $x,y\in \N$.
We have the induced cartesian square
\[
\begin{tikzcd}
S'\times \pt_{\N^2}\ar[d,"u'"']\ar[r,"g''"]&
S\times \pt_{\N^2}\ar[d,"u"]
\\
S'\times \pt_{\N^2}\ar[r,"g'"]&
S\times \pt_{\N^2},
\end{tikzcd}
\]
where $u:=\id \times \pt_\theta$.
Since $u^*$ and $u'^*$ are equivalence of $\infty$-categories by Theorems \ref{homeomorphism.5} and \ref{homeomorphism.9},
it suffices to show that the natural transformation $g^*(fu)_*\xrightarrow{Ex} (f'u')_*g''^*$ is an isomorphism.
This is done in Lemma \ref{base2.1}.
\end{proof}

\begin{lem}
\label{base2.3}
Let $g\colon S'\to S$ be a morphism in $\lSch/B$.
Consider the induced cartesian square
\[
C:=
\begin{tikzcd}
S'\times \pt_\N\ar[d,"f'"']\ar[r,"g'"]&
S\times \pt_\N\ar[d,"f"]
\\
S'\ar[r,"g"]&
S,
\end{tikzcd}
\]
where $f$ is the projection.
Then the natural transformation
\[
Ex\colon g^*f_*
\to
f_*'g'^*
\]
is an isomorphism.
\end{lem}
\begin{proof}
By \cite[Proposition 3.6.9]{logshriek},
$f'^*$ is conservative.
It suffices to show that the natural transformation $f'^*g^*f_*\xrightarrow{Ex} f'^*f_*'g'^*$ is an isomorphism.
By
\cite[Corollary 3.7.6]{logshriek},
we reduce to Lemma \ref{base2.2}.
\end{proof}

\begin{lem}
\label{base2.4}
Let $\theta\colon P\to Q$ be a local saturated homomorphism of sharp fs monoids,
and let $g\colon S'\to S$ be a morphism in $\lSch/B$ such that $S\simeq \ul{S}\times \pt_P$.
Consider the induced cartesian square
\[
\begin{tikzcd}
S'\times_{\pt_P} \pt_Q\ar[d,"f'"']\ar[r,"g'"]&
S\times_{\pt_P} \pt_Q\ar[d,"f"]
\\
S'\ar[r,"g"]&
S.
\end{tikzcd}
\]
Then the natural transformation
\[
Ex\colon g^*f_* \to f_*'g'^*
\]
is an isomorphism.
\end{lem}
\begin{proof}

Let $G$ be a $\theta$-maximal critical face of $Q$.
By toric resolution of singularities \cite[Theorem 11.1.9]{CLStoric},
there exists a free submonoid $M\simeq \N^n$ of $G$ for some integer $n\geq 0$ such that $M^\gp\to G^\gp$ is an isomorphism.
The induced homomorphism $P^\gp\oplus G^\gp\to Q^\gp$ is an isomorphism since it is injective by \cite[Theorem I.4.8.14(7)]{Ogu} and surjective by \cite[Lemma I.4.7.8, Theorem I.4.8.14(6)]{Ogu}.
Consider the induced factorization
\[
\ul{S}\times \pt_Q
\xrightarrow{v}
\ul{S}\times \pt_{P\oplus M}
\xrightarrow{u}
\ul{S}\times \pt_P
\]
of $f$.
Observe that $v$ is a virtual isomorphism.
Since $u$ and $f$ are saturated,
we have
\[
\ul{S'\times_{\pt_P} \pt_Q}
\simeq
\ul{S'}
\simeq
\ul{S'\times_{\pt_P} \pt_{P\oplus M}}.
\]
Hence $\ul{v'}$ is an isomorphism,
where $v'\colon S'\times_{\pt_P} \pt_Q \to S'\times_{\pt_P} \pt_{P\oplus M}$ is the pullback of $v$.
It follows that $v'$ is a virtual isomorphism.
By Theorems \ref{homeomorphism.5} and \ref{homeomorphism.9},
$v^*$ and $v'^*$ are equivalences of $\infty$-categories.
Hence we can replace $f$ with $u$ to reduce to the case when $\theta$ is the inclusion $P\to P\oplus \N^n$.
Use Lemmas \ref{base2.6} and \ref{base2.3} repeatedly to finish the proof.
\end{proof}

\begin{prop}
\label{base2.5}
Let
\[
C:=
\begin{tikzcd}
X'\ar[d,"f'"']\ar[r,"g'"]&
X\ar[d,"f"]
\\
S'\ar[r,"g"]&
S
\end{tikzcd}
\]
be a cartesian square such that $f$ is $\Q$-integral and proper.
Then the natural transformation
\[
Ex\colon
g^*f_*
\to
f_*'g'^*
\]
is an isomorphism.
\end{prop}
\begin{proof}
By \cite[Proposition 2.4.3]{logshriek},
we have a Kummer log smooth morphism $U\to S$ such that the projection $X\times_S U\to U$ is saturated and $v^*$ is conservative,
where $v\colon S'\times_S U\to S'$ is the projection.
It suffices to show that $v^*g^*f_*\xrightarrow{Ex} v^*f_*'g'^*$ is an isomorphism.
By ($\eSm$-BC),
we can replace $C$ by $C\times_X U$ to reduce to the case when $f$ is saturated proper.

Let $\cF$ be the class of morphisms $f$ such that $g^*f_*
\xrightarrow{Ex} f_*'g'^*$ is always an isomorphism.
We have the induced factorization
\[
X\xrightarrow{q}\ul{X}\times_{\ul{S}} S \xrightarrow{p} S
\]
of $f$.
Every strict proper morphism is in $\cF$ by
\cite[Proposition 2.4.13]{logshriek}.
In particular,
$p$ is in $\cF$.
Since $\cF$ is closed under compositions by Lemma \ref{base2.6},
it suffices to show that $q$ is in $\cF$.
Hence we reduce to the case when $f$ is a saturated proper morphism such that $\ul{f}$ is an isomorphism.

Let $\{W_1,\ldots,W_n\}$ be a Zariski covering on $X$.
By
\cite[Proposition 2.4.5]{logshriek} and induction on $n$,
it suffices to show that $g^*f_*w_\sharp \xrightarrow{Ex} f_*'g'^*w_\sharp$ is an isomorphism whenever $w$ is the open immersion from an intersection $W:=W_{i_1}\cap \cdots \cap W_{i_r}$ to $X$.
Since $\ul{f}$ is an isomorphism,
we have the open subscheme $V$ of $S$ such that $W\simeq V\times_S X$.
By ($\eSm$-BC) and (Supp),
it suffices to show the claim for $C\times_S V$.
Hence we reduce to the case when $f$ is a saturated proper morphism with a chart such that $\ul{f}$ is an isomorphism.

Let $Z\to S$ be a strict closed immersion with its open complement $U\to S$.
Consider the pullbacks $i\colon Z\times_S X\to X$ and $j\colon U\times_S X\to X$.
By (Loc),
it suffices to show that $g^*f_* i_*\xrightarrow{Ex} f_*'g'^*i_*$ and $g^*f_*j_\sharp \xrightarrow{Ex} f_*'g'^*j_\sharp$ are isomorphisms.
Furthermore,
\cite[Proposition 2.4.13]{logshriek} (resp.\ ($\eSm$-BC) and (Supp)) implies that to show that the first (resp.\ second) natural transformation is an isomorphism,
it suffices to show the claim for $C\times_S Z$ (resp.\ $C\times_S U$).
Hence it suffices to show the claim after a stratification on $\ul{S}\simeq \ul{X}$.
Together with Construction \ref{affine.5},
we reduce to the case when $f$ is the morphism $\id \times \pt_\theta\colon \ul{S}\times \pt_Q \to \ul{S} \times \pt_P$ for some saturated local homomorphism of fs monoids $\theta\colon P\to Q$.
This case is done in Lemma \ref{base2.4}.
\end{proof}

\begin{thm}
\label{base2.7}
There exists a functor
\begin{gather*}
(\sT^{ex})_!^*
\colon
\Corr(\lSch/B,\Q\textup{-integral compactifiable},\mathrm{all})
\to
\infCat_\infty
\end{gather*}
sending a morphism $(Y\xleftarrow{g} X \xrightarrow{f} S)$ in the correspondence category to
\[
g_!f^*
\colon
\sT^{ex}(S)\to \sT^{ex}(Y).
\]
\end{thm}
\begin{proof}
This is an immediate consequence of Proposition \ref{base2.5} and \cite[Proposition 3.7.7]{logshriek}.
\end{proof}

\begin{cor}
\label{base2.8}
Let
\[
\begin{tikzcd}
X'\ar[d,"f'"']\ar[r,"g'"]&
X\ar[d,"f"]
\\
S'\ar[r,"g"]&
S
\end{tikzcd}
\]
be a cartesian square in $\lSch/B$ such that $f$ is $\Q$-integral and separated.
Then there is a natural isomorphism
\[
g^*f_!
\simeq
f_!'g'^*.
\]
\end{cor}
\begin{proof}
This is an immediate consequence of Theorem \ref{base2.7}.
\end{proof}

\begin{rmk}
\label{functorial.2}
In \eqref{functorial.1.1}, the horizontal arrows are isomorphisms by Remark \ref{functorial.8} and Corollary \ref{base2.8}.
\end{rmk}

\subsection{Proof of Poincar\'e duality}
\label{AQAP}

In this section,
we finish the proof of Poincar\'e duality.

\begin{lem}
\label{AQAP.1}
If $S\in \lSch/B$, then we have
\[
f_!f^*\simeq 0,
\]
where $f\colon S\times \A_\N\to S$ is the projection.
\end{lem}
\begin{proof}
Let $j\colon S\times \A_\N\to S\times (\P^1,0)$ be the obvious inclusion, and let $g\colon S\times (\P^1,0)\to S$ be the projection.
We need to show $g_*j_!j^*g^*\simeq 0$.
By (Loc), it suffices to show that the natural transformation $g_*g^* \xrightarrow{ad} g_*i_*i^*g^*$ is an isomorphism,
where $i\colon S\to S\times (\P^1,0)$ is the strict closed immersion at $\infty$.
This follows from ($\boxx$-inv) since $gi$ is an isomorphism.
\end{proof}

\begin{lem}
\label{AQAP.2}
Let $P$ be an fs monoid, and let $F$ be a nontrivial face of $P$.
If $S\in \lSch/B$, then we have
\[
f_!f^*\simeq 0,
\]
where $f\colon S\times \A_{P,P-F}\to S$ is the projection.
\end{lem}
\begin{proof}
We proceed by induction on the rank $r$ of $\ol{P}^\gp$.
If $r=1$, then $P\simeq F\simeq \N$.
Use Lemma \ref{AQAP.1} to show $f_!f^*\simeq 0$.

Assume $r>1$.
Let $G$ be a $1$-dimensional face of $F$, and choose generators $a_1,\ldots,a_n$ of the ideal $P-G$ of $P$.
For all nonempty subset $I:=\{a_{i_1},\ldots,a_{i_d}\}$ of $\{1,\ldots,n\}$, let $Q_I$ be the localization $Q_{\langle a_{i_1},\ldots,a_{i_d}\rangle}$, and let $F_I$ be the face of $P_I$ generated by $F$.
Consider the projection
\[
f_I\colon S\times \A_{Q_I,Q_I-F_I}\to S.
\]
Since $\rank(\ol{Q_I}^\gp)<r$, we have $f_{I!}f_I^*\simeq 0$ by induction.
As a consequence of \eqref{affine.2.1}, we see that the family
\[
\{S\times \A_{P_{\{1\}},P_{\{1\}}-F_{\{1\}}},
\ldots,
S\times \A_{P_{\{n\}},P_{\{n\}}-F_{\{n\}}}\}
\]
is a Zariski cover of $S\times (\A_{P,P-F}-\A_{P,P-G})$.
Hence we obtain $g_!g^*\simeq 0$ by Zariski descent and induction, where
\[
g\colon S\times (\A_{P,P-F}-\A_{P,P-G})\to S
\]
is the projection.
Let
\[
h\colon S\times \A_{P,P-G}\to S
\]
be the projection.
By (Loc), to show $f_!f^*\simeq 0$, it suffices to show $h_!h^*\simeq 0$.

By toric resolution of singularities \cite[Theorem 11.1.9]{CLStoric}, there exists a submonoid $Q$ of $P$ such that $Q^\gp\to P^\gp$ is an isomorphism, $Q\simeq \N^d$ for some integer $d$, and $Q$ contains $G$ as a face.
Use Propositions \ref{affine.1} and \ref{affine.4} to see that the induced morphism
\[
u\colon S\times \A_{P,P-G}\to S\times \A_{Q,Q-G}
\]
is a virtual isomorphism.
Theorems \ref{homeomorphism.5} and \ref{homeomorphism.9} imply that $u^*$ is an equivalence of $\infty$-categories.
Hence it suffices to show $v_!v^*\simeq 0$, where
\[
v\colon S\times \A_{Q,Q-G}\to S
\]
is the projection.
We have $\A_{Q,Q-G}\simeq \pt_{\N^{d-1}}\times \A_\N$.
Let
\[
w\colon S\times \pt_{\N^{d-1}}\times \A_\N\to S\times \pt_{\N^{d-1}}
\]
be the projection.
Lemma \ref{AQAP.1} shows $w_!w^*\simeq 0$.
This implies $v_!v^*\simeq 0$.
\end{proof}

\begin{lem}
\label{AQAP.3}
Let $\theta\colon P\to Q$ be a local vertical saturated homomorphism of sharp fs monoids, and let $G$ be a nontrivial $\theta$-critical face \cite[Definition I.4.2.10]{Ogu} of $Q$.
If $S\to \pt_P$ is a strict morphism in $\lSch/B$, then we have
\[
f_!f^*\simeq 0,
\]
where $f\colon S\times_{\A_P}\A_{Q,Q-G}\to S$ is the projection.
\end{lem}
\begin{proof}
Let $F$ be a maximal $\theta$-critical face of $Q$ containing $G$.
Consider the induced morphism
\[
g\colon 
S\times_{\A_P}\A_{Q,Q-G}
\to
S\times_{\A_P}\A_{P\oplus F,P\oplus F-G}
\]
Use Proposition \ref{affine.4} to obtain the stratifications
\[
\coprod_{H\subset G} \A_{Q_H,Q_H^+}
\text{ and }
\coprod_{H\subset G} \A_{(P\oplus F)_H,(P\oplus F)_H^+}
\]
on $\A_{Q,Q-G}$ and $\A_{P\oplus F,P\oplus F-G}$,
where the coproduct runs over faces $H$ of $G$.
As a consequence of \cite[Theorem 4.8.14(6)]{Ogu}, we have $P^\gp \oplus F^\gp \simeq Q^\gp$.
After taking quotients, we have $((P\oplus F)/H)^\gp \simeq (Q/G)^\gp$.
Together with Proposition \ref{affine.1}, we see that the induced morphism
\[
\A_{Q_H,Q_H^+}
\to
\A_{(P\oplus F)_H,(P\oplus F)_H^+}
\]
is a virtual isomorphism.
It follows that $g$ is a virtual isomorphism.
Theorems \ref{homeomorphism.5} and \ref{homeomorphism.9} imply that $g^*$ is an equivalence of $\infty$-categories.

Hence it suffices to show $h_!h^*\simeq 0$, where $h\colon S\times_{\A_P}\A_{P\oplus F,P\oplus F-G}\simeq S\times \A_{F,F-G}\to S$ is the projection.
This follows from Lemma \ref{AQAP.3}.
\end{proof}

\begin{lem}
\label{AQAP.5}
Let $\theta\colon P\to Q$ be a local vertical saturated homomorphism of sharp fs monoids, and let $G$ be a maximal $\theta$-critical face of $Q$.
Assume that $\theta$ is not an isomorphism.
If $S\to \pt_P$ is a strict morphism in $\lSch/B$, then the natural transformation
\[
f_!j_\sharp j^* f^*\xrightarrow{ad'}
f_!f^*
\]
is an isomorphism, 
where $f\colon S\times_{\A_P}\A_Q\to S$ is the projection, and $j\colon S\times_{\A_P}\A_{Q_G}\to S\times_{\A_P}\A_Q$ is the open immersion.
\end{lem}
\begin{proof}
Let $G_1:=G,G_2,\ldots,G_r$ be the maximal $\theta$-critical faces of $Q$.
For every nonempty subset $I:=\{i_1,\ldots,i_s\}\subset \{2,\ldots,r\}$, we set
\[
G_I:=G_{i_1}\cap \cdots \cap G_{i_s}.
\]
Let $f_I\colon S\times_{\A_P}\A_{Q,Q-G_I}\to S$ be the projection.
The family
\[
\{\pt_P\times_{\A_P}\A_{Q,Q-G_1},\ldots,\pt_P\times_{\A_P}\A_{Q,Q-G_r}\}
\]
is a closed cover of $\pt_P\times_{\A_P}\A_Q$.
We set $K:=G_1\cup \cdots \cup G_r$.
Due to \cite[Theorem I.4.8.14(7)]{Ogu}, the summation map
\[
P\times K \to Q
\]
is bijective.
If $r=1$, then $Q\simeq P\oplus G_1$, so $\theta$ is not vertical.
Hence $r>1$.
The bijection $P\times K\to Q$ gives a homeomorphism between $K\otimes \R$ and $(Q/P)\otimes \R$.
Since $\theta$ is vertical, $Q/P$ is a group.
Hence $K\otimes \R$ is homeomorphic to $\R^n$ for some integer $n\geq 1$.

With this homeomorphism, we can regard $G_1\otimes \R$ as a sharp cone in $\R^n$.
After a linear transformation, we can regard $G_1\otimes \R$ as $\R_{\geq 0}^n$.
From this, we see that every proper face of $G_1$ is contained in $G_2\cup \cdots \cup G_r$.
Together with Proposition \ref{affine.4}, we deduce that the family
\[
\{S\times_{\A_P}\A_{Q,Q-G_2},\ldots,S\times_{\A_P}\A_{Q,Q-G_r}\}
\]
is a closed cover of $S\times_{\A_P}(\A_Q-\A_{Q_G})$.
It follows that we have
\begin{equation}
\label{AQAP.3.1}
f_!f^*
\simeq
\lim_{I\in \cP}f_{I!}f_I^*,
\end{equation}
where $\cP$ is the category of nonempty subsets of $\{2,\ldots,r\}$ whose morphisms are the inclusions.

By \cite[Corollary 1.4.4.6]{HA}, the object $\unit\in \sT(S)$ gives a colimit preserving functor $\alpha\colon \Sp \to \sT(S)$ sending the sphere spectrum $\mathbb{S}$ to $\unit$.
For a space $V$, we set
\[
\epsilon(V)
:=
\alpha(\map(\Sigma^\infty V,\mathbb{S}))
\]
for contravariant functoriality, where $\map$ denotes the mapping spectrum.
Let $g\colon S\times_{\A_P}\pt_Q\to S$ be the projection.
We have
\[
\epsilon(G_I\otimes \R-0)
\simeq
\left\{
\begin{array}{ll}
\unit & \text{if }G_I\neq 0,
\\
0 & \text{if }G_I=0.
\end{array}
\right.
\]
Hence Lemma \ref{AQAP.3} can be written as
\[
f_{I!}f_I^*\cF
\simeq
g_!g^*\cF \otimes \fib(\unit\to \epsilon(G_I\otimes \R-0))
\]
that is functorial in $I$.
We set $W:=K\otimes \R -((G\otimes \R)^\circ\cup 0)$, where $(-)^\circ$ denotes the interior.
Since $\{G_1\otimes \R-0,\ldots,G_r\otimes \R-0\}$ is a closed cover of $W$, Mayer-Vietoris and \eqref{AQAP.3.1} give
\[
f_!f_!^*\cF
\simeq
g_!g^*\cF \otimes \fib(\unit \to \epsilon(W)).
\]
To conclude, observe that $W\simeq \R^n-(\R_{>0}^n \cup 0)$ is contractible.
\end{proof}

\begin{lem}
\label{AQAP.6}
Let $\theta\colon P\to Q$ be a local vertical saturated homomorphism of sharp fs monoids, and let $S\to \A_P$ be a strict morphism.
Then the natural transformation
\[
f_\sharp f^*
\xrightarrow{\mathfrak{p}_f^n}
f_!\Sigma_f^n f^*
\]
is an isomorphism, where $f\colon S\times_{\A_P}\A_Q\to S$ is the projection.
\end{lem}
\begin{proof}
The claim is trivial if $\theta$ is an isomorphism,
so assume that $\theta$ is not an isomorphism.
Let $G$ be a maximal $\theta$-critical face of $Q$,
and let $j\colon S\times_{\A_P}\A_{Q_G}\to S\times_{\A_P}\A_Q$ be the induced open immersion.
As observed in the proof of the implication (5)$\Rightarrow$(6) in \cite[Theorem I.4.8.14]{Ogu},
the induced homomorphism $P\to Q/G$ is an isomorphism.
This implies that $fj$ is strict.
The diagram
\[
\begin{tikzcd}
(fj)_\sharp j^* f^*\ar[rr,"\simeq"]\ar[d,"\mathfrak{p}_{fj}^n"']\ar[rrd,"(a)" description,phantom]&
&
f_\sharp j_\sharp j^*f^*\ar[d,"\mathfrak{p}_f^n"]\ar[r,"ad'"]&
f_\sharp f^*\ar[d,"\mathfrak{p}_f^n"]
\\
(fj)_!\Sigma_{fj}^n j^*f^*\ar[r,"\simeq"]&
f_!j_\sharp\Sigma_{fj}^nj^* f^*\ar[r,"Ex"]&
f_!\Sigma_{f}^n j_\sharp j^*f^*\ar[r,"ad'"]&
f_!\Sigma_f^n f^*.
\end{tikzcd}
\]
commutes since the diagram surrounding $(a)$ commutes by Proposition \ref{functorial.3}.
The lower middle horizontal arrow is an isomorphism as observed in Remark \ref{functorial.4}.
The left vertical arrow is an isomorphism by Proposition \ref{puritytran.1}.
The upper right horizontal arrow is an isomorphism by \cite[Proposition 2.4.2]{logshriek}.
Since $\Omega_{S\times_{\A_P}\A_Q/S}^1$ is free by \cite[Propositions IV.1.1.4(2), IV.1.2.15(2)]{Ogu},
we have $\Sigma_{f}^n\simeq (d)[2d]$ for some integer $d$.
Together with Lemma \ref{AQAP.5},
we see that the lower right horizontal arrow is an isomorphism.
Hence the right vertical arrow is an isomorphism.
\end{proof}

To use induction later, we consider the following condition for all integer $d\geq 0$:
\begin{itemize}
\item $(\mathrm{Pur}_d)$ If $f$ is a separated vertical exact log smooth morphism $f\colon X\to S$ in $\lSch/B$ such that the inequality
\[
\rank(\cM_{X/S,x}^\gp) \leq d
\]
is satisfied for all point $x\in X$, then $f$ is pure.
\end{itemize}

\begin{prop}
\label{AQAP.4}
Let $\theta\colon P\to Q$ be a local vertical saturated homomorphism of sharp fs monoids, and let $S\to \A_P$ be a strict morphism.
Assume that $(\mathrm{Pur}_{d-1})$ is satisfied, where $d:=\rank(Q^\gp)-\rank(P^\gp)$.
Then the projection $f\colon X:=S\times_{\A_P}\A_Q\to S$ is pure.
\end{prop}
\begin{proof}
The quasi-coherent sheaf $\Omega_{X/S}^1$ is a free $\cO_X$-module of rank $d$ by \cite[Propositions IV.1.1.4(2), IV.1.2.15(2)]{Ogu}.
It follows that $\Sigma_f^n$ is isomorphic to $\tau:=(d)[2d]$.

We need to show that the natural transformation $\mathfrak{p}_f^n\colon f_\sharp \to f_!\tau$ admits left and right inverses.
Consider the composition
\[
\alpha\colon f_!\tau \xrightarrow{ad} f_!f^*f_\sharp 
\tau \xrightarrow{(\mathfrak{p}_f^n)^{-1}} f_\sharp f^* f_\sharp \xrightarrow{ad'} f_\sharp,
\]
where the second arrow is defined and an isomorphism by Proposition \ref{AQAP.6}.
The commutative diagram
\[
\begin{tikzcd}
f_\sharp\ar[d,"\mathfrak{p}_f^n"']\ar[r,"ad"]&
f_\sharp f^*f_\sharp \ar[rrd,"\simeq",bend left]\ar[d,"\mathfrak{p}_f^n"']
\\
f_!\tau\ar[r,"ad"]&
f_! f^*f_\sharp\tau \ar[rr,"(\mathfrak{p}_f^n)^{-1}"]&&
f_\sharp f^*f_\sharp\ar[r,"ad'"]&
f_\sharp
\end{tikzcd}
\]
shows that $\alpha$ is a left quasi-inverse of $\mathfrak{p}_f^n\colon f_\sharp \to f_!\tau$.

Let $j\colon Z:=S\times_{\A_P} (\A_Q-\A_{Q,Q^+})\to X$ and $i\colon U:=S\times_{\A_P} \A_{Q,Q^+}\to X$ be the obvious strict immersions.
We set $g:=fi$.
Observe that the morphism of underlying schemes $\ul{g}\colon \ul{Z}\to \ul{S}$ is an isomorphism.
Using (Loc),
we have a commutative diagram
\begin{equation}
\label{AQAP.4.1}
\begin{tikzcd}
f_\sharp j_\sharp j^*\ar[r,"ad'"]\ar[d,"\mathfrak{p}_f^n"']&
f_\sharp \ar[r,"ad"]\ar[d,"\mathfrak{p}_f^n"]&
f_\sharp i_*i^*\ar[d,"\mathfrak{p}_f^n"]
\\
f_!j_\sharp j^*\tau\ar[r,"ad'"]&
f_!\tau\ar[r,"ad'"]&
f_!i_*i^*\tau
\end{tikzcd}
\end{equation}
whose rows are cofiber sequences.
By Proposition \ref{functorial.3} and Remark \ref{functorial.4}, $(\mathrm{Pur}_{d-1})$ implies that the left vertical arrow of \eqref{AQAP.4.1} is an isomorphism.
Hence we only need to show that the natural transformation
\begin{equation}
\label{AQAP.4.2}
f_\sharp i_*\xrightarrow{\mathfrak{p}_f^n} g_*\tau
\end{equation}
is an isomorphism.
Since it already has a left quasi-inverse, it suffices to construct its right quasi-inverse.

We have the commutative diagram
\begin{equation}
\label{AQAP.4.3}
\begin{tikzcd}
f_\sharp j_\sharp j^*f^*\ar[r,"ad'"]\ar[d,"\mathfrak{p}_f^n"']&
f_\sharp f^*\ar[r,"ad"]\ar[d,"\mathfrak{p}_f^n"]&
f_\sharp i_*i^*f^*\ar[d,"\mathfrak{p}_f^n"]
\\
f_!j_\sharp j^*f^*\tau\ar[r,"ad'"]&
f_!f^*\tau\ar[r,"ad'"]&
f_!i_*i^*f^*\tau
\end{tikzcd}
\end{equation}
whose rows are cofiber sequences.
We already know that the left vertical arrow of \eqref{AQAP.4.3} is an isomorphism.
The middle vertical arrow of \eqref{AQAP.4.3} is an isomorphism by Lemma \ref{AQAP.6}.
Hence the right vertical arrow of \eqref{AQAP.4.3} is an isomorphism.

Consider the composition
\[
\beta\colon 
g_* \tau
\xrightarrow{ad}
g_*g^*g_*\tau 
\xrightarrow{(p_f^n)^{-1}}
f_\sharp i_*g^*g_*\tau
\xrightarrow{ad'}
f_\sharp i_*,
\]
where the second arrow is defined and an isomorphism by the above paragraph.
The commutative diagram
\[
\begin{tikzcd}
g_*\tau\ar[r,"ad"]&
g_*g^*g_* \tau\ar[rrr,"(\mathfrak{p}_f^n)^{-1}"]\ar[rrrd,"\simeq"']&&&
f_\sharp i_*g^*g_*\tau\ar[r,"ad'"]\ar[d,"\mathfrak{p}_f"]&
f_\sharp i_*\ar[d,"\mathfrak{p}_f"]
\\
&
&&&
g_* g^*g_*\tau\ar[r,"ad'"]&
g_*\tau
\end{tikzcd}
\]
shows that \eqref{AQAP.4.2} admits a right quasi-inverse.
\end{proof}

Now, we prove Poincar\'e duality:

\begin{thm}
\label{poincare.1}
Let $f\colon X\to S$ be a separated vertical exact log smooth morphism in $\lSch/B$.
Then $f$ is pure.
\end{thm}
\begin{proof}
We need to show $(\mathrm{Pur}_d)$ for every integer $d\geq 0$.
We proceed by induction on $d$.
Assume that the inequality $\rank(\cM_{X/S,x}^\gp)\leq d$ is satisfied for all point $x\in X$.
If $d=0$, then $f$ is strict.
Hence Proposition \ref{puritytran.1} implies $(\mathrm{Pur}_0)$.
From now,
we assume $d>0$ and $(\mathrm{Pur}_{d-1})$.

(I) \emph{Saturation of $f$.}
Using \cite[Proposition 2.4.3]{logshriek},
we have a Kummer log smooth morphism $g\colon S'\to S$ such that $g^*$ is conservative and the projection $X\times_S S'\to S'$ is saturated.
By Proposition \ref{functorial.1} and Remark \ref{functorial.2}, we reduce to showing that the projection $X\times_S S'\to S'$ is pure.
Hence we reduce to the case when $f$ is saturated.

(II) \emph{Locality on $S$.}
Suppose that $i\colon Z\to S$ is a strict closed immersion and $j\colon U\to S$ is its open complement.
By (Loc),
the pair of functors $(i^*,j^*)$ is conservative.
As in (I), we reduce to showing that the projections $X\times_S Z\to Z$ and $X\times_S U\to U$ are pure.
Together with Propositions \ref{affine.1} and \ref{affine.4}, we reduce to the case when there is a strict morphism $S\to \pt_P$ for some sharp fs monoid $P$.

(III) \emph{Locality on $X$.}
Suppose that $\{v_i\colon V_i\to X\}_{i\in I}$ is a Zariski covering family.
By Proposition \ref{functorial.6}, we reduce to showing that each $fv_i$ is pure.
Together with \cite[Propositions A.3, A.4]{divspc}, we reduce to the case when $f$ admits a neat chart $\theta\colon P\to Q$ and $P$ and $Q$ are sharp.

(IV) \emph{Final step of the proof.}
As observed in \cite[Proposition A.4]{divspc},
$f$ admits the factorization
\[
X\xrightarrow{g} S\times_{\A_P}\A_Q \xrightarrow{p} S
\]
such that $g$ is strict \'etale and $p$ is the projection.
Since $f$ is saturated and vertical, $\theta$ is saturated and vertical.
Hence Proposition \ref{AQAP.4} implies that $p$ is pure.
By Proposition \ref{puritytran.1},
$g$ is pure.
Together with Proposition \ref{functorial.5},
we deduce that $f=pg$ is pure.
\end{proof}

\subsection{On Lefschetz duality}
\label{Lefschetz}

Poincar\'e duality in algebraic topology does not hold for compact manifolds with non-empty boundary, and Lefschetz duality is for such a manifold with boundary.
In this section,
we prove a motivic analogue of Lefschetz duality when the codomain has trivial log structure.

Let us first collect several morphisms $f$ such that $f_!$ admits a right adjoint $f^!$.
If we assume that $M_S(X)$ is compact in $\sT^{ex}(X)$ for every vertical exact log smooth morphism $X\to S$,
then $f^!$ always exists by Proposition \ref{Verdier.9}.
Hence the following discussions of $f^!$ are only meaningful if we instead assume that $\sT^{ex}$ satisfies strict \'etale descent.

If $f$ is a separated vertical exact log smooth morphism,
$f^!$ exists by Theorem \ref{poincare.1}.
If $f$ is a strict closed immersion,
then $f^!$ exists by (Loc).

\begin{lem}
\label{shriek.1}
For $S\in \lSch/B$,
consider the projection $f\colon S\times\A_\N \to S$.
Then $f_!$ admits a right adjoint $f^!$.
\end{lem}
\begin{proof}
We need to show that $f_!$ preserves colimits.
Consider the projection $g\colon S\times \pt_\N\to S$.
By \cite[Proposition 3.6.9]{logshriek},
$g^*$ is conservative.
Hence it suffices to show that $g^*f_!$ preserves colimits.
By \cite[Proposition 3.7.1]{logshriek},
it suffices to show that $f_!'$ preserves colimits,
where $f'\colon S\times \pt_\N \times \A_\N\to S\times \pt_\N$ is the projection.

Consider the virtual isomorphism
\[
h\colon S\times \pt_\N \times \A_\N
\to
S\times \pt_\N \times \A_\N
\]
induced by the homomorphism
\[
\N\oplus \N \to \N\oplus \N,
\;
(a,b) \mapsto (a+b,b).
\]
By Theorems \ref{homeomorphism.5} and \ref{homeomorphism.9},
it suffices to show that $f_*'h_*$ preserves colimits.

We have a factorization
\[
S\times \pt_\N \times \A_\N
\xrightarrow{i}
S\times \pt_{\N} \times_{\A_\N} \A_{\N\oplus \N}
\xrightarrow{u}
S\times \pt_\N
\]
of $f'h$,
where $i$ is a strict closed immersion,
and $u$ is induced by the diagonal homomorphism $\N\to \N\oplus \N$.
Since $u$ is a separated vertical exact log smooth morphism,
$u_!$ preserves colimits by Theorem \ref{poincare.1}.
To conclude,
observe that $i_*$ preserves colimits by (Loc).
\end{proof}

\begin{lem}
\label{shriek.2}
Let $f\colon X\to S$ be a log smooth morphism in $\lSch/B$.
If $S$ has trivial log structure,
then $f_!$ admits a right adjoint $f^!$.
\end{lem}
\begin{proof}
We can work Zariski locally on $X$.
By \cite[Proposition A.4]{divspc},
we may assume that $X$ admits a chart $P$ such that $P$ is sharp and the induced morphism $g\colon X\to S\times \A_P$ is strict smooth.
Since $g_!$ preserves colimits by Theorem \ref{poincare.1},
it suffices to show that $p_!$ preserves colimits,
where $p\colon S\times \A_P\to S$ is the projection.

By toric resolution of singularities \cite[Theorem 11.1.9]{CLStoric},
there exists a dividing cover $h\colon S\times \T_\Sigma \to S\times \A_P$ such that $\Sigma$ is a smooth fan.
Since $\id\simeq h_*h^*$,
it suffices to show that $(ph)_!$ preserves colimits.
We can also work Zariski locally on $S\times \T_\Sigma$.
Hence we reduce to showing that for every integer $d\geq 0$,
$q_*$ preserves colimits,
where $q\colon S\times \A_{\N^d}\to S$ is the projection.
Apply Lemma \ref{shriek.1} repeatedly to conclude.
\end{proof}

Now,
we begin with the proof of Lefschetz duality for the case where the codomain has trivial log structure.

\begin{lem}
\label{functorial.9}
Let $f\colon X\to S$ be a separated strict smooth morphism in $\lSch/B$.
Then the diagram
\[
\begin{tikzcd}
f_\sharp f^*\ar[r,"ad'"]\ar[d,"\mathfrak{p}_f (\mathfrak{q}_f)^{-1}"']&
\id\ar[d,"ad'",leftarrow]
\\
f_!\Sigma_f \Omega_f f^!
\ar[r,"ad'"]&
f_!f^!
\end{tikzcd}
\]
commutes.
\end{lem}
\begin{proof}
Let $p_1,p_2\colon X\times_S X\to X$ be the two projections,
and let $a\colon X\to X\times_S X$ be the diagonal morphism.
The diagram
\[
\begin{tikzcd}
f_\sharp f^*
\ar[rrr,"ad'"]
\ar[rrd,"\id"]
\ar[d,"\simeq"']
&
&
&
\id\ar[d,"\id",leftarrow]
\\
f_\sharp p_{1!}a_*a^!p_1^!f^*
\ar[r,"ad'"']
\ar[d,"ExEx"']
&
f_\sharp p_{1!}p_1^!f^*
\ar[r,"ad'"']
\ar[d,"ExEx"]
&
f_\sharp f^*
\ar[r,"ad'"']
&
\id
\ar[d,"ad'",leftarrow]
\\
f_! p_{2\sharp} a_* a^! p_2^* f^!
\ar[r,"ad'"]
&
f_!p_{2\sharp}p_2^*f^!
\ar[rr,"ad'"]
&
&
f_!f^!
\end{tikzcd}
\]
commutes,
which shows the claim.
\end{proof}

\begin{lem}
\label{Lefschetz.5}
For $S\in \lSch/B$,
consider the projection $f\colon S\times \A_\N \to S$ and the induced open immersion $j\colon S\times \G_m\to S\times \A_\N$.
Then the natural transformation
\[
f_!j_! j^!f^!
\xrightarrow{ad'}
f_!f^!
\]
is an isomorphism.
\end{lem}
\begin{proof}
Let $g\colon S\times \square\to S$ be the projection,
and let $v\colon S\times \A^1\to S\times \square$ be the induced open immersion.
By Zariski descent,
it suffices to show that the natural transformation
\[
g_!v_!v^!g^!\xrightarrow{ad'} g_!g^!
\]
is an isomorphism.
The natural transformation $g_!g^!\xrightarrow{ad'}\id$ is an isomorphism since its left adjoint $\id\xrightarrow{ad} g_*g^*$ is an isomorphism by ($\square$-inv).
Hence it suffices to show that the natural transformation
\[
g_!v_!v^!g^!\xrightarrow{ad'}\id
\]
is an isomorphism.
By Proposition \ref{puritytran.1} and Lemma \ref{functorial.9},
it suffices to show that the natural transformation
\[
(gv)_\sharp (gv)^*
\xrightarrow{ad'}
\id
\]
is an isomorphism.
This holds by ($\A^1$-inv).
\end{proof}

\begin{lem}
\label{Lefschetz.6}
Let $f\colon X\to S$ be a log smooth morphism in $\lSch/B$ that admits a factorization
\[
X\xrightarrow{q} S\times \A_{\N^d} \xrightarrow{p} S
\]
such that $p$ is the projection, $q$ is strict \'etale, and $d$ is a nonnegative integer.
Let $j\colon X-\partial_S X\to X$ be the obvious open immersion.
Then the natural transformation
\[
f_*j_!j^!f^!
\xrightarrow{ad'}
f_*f^!
\]
is an isomorphism.
\end{lem}
\begin{proof}
We proceed by induction on $d$.
The claim is trivial if $d=0$,
so assume $d>0$.
Let $p'\colon S\times \A_{\N^d}\to S\times \A_\N$ be the morphism induced by the first inclusion $\N\to \N^d$,
and let $j'\colon X-\partial_{S\times \A_{\N}}X\to X$ be the obvious open immersion.
By induction,
the natural transformation
\[
p_*'q_*j_!'j'^!q^!p'^!
\xrightarrow{ad'}
p_*'q_*q^!p'^!
\]
is an isomorphism.
Hence it suffices to show that the natural transformation
\[
f_*j_!j^!f^!
\xrightarrow{ad'}
f_*j_!'j'^!f^!
\]
is an isomorphism.
Replace $X\to S$ with $X-\partial_{S\times \A_{\N}}X\to S\times \A_\N \times \G_m^{d-1}$ to reduce to the case when $d=1$.

Consider the induced cartesian square
\[
\begin{tikzcd}
X\times_{\A_\N}\pt_\N\ar[r,"i"]\ar[d,"q'"']&
X\ar[d,"q"]
\\
S\times \pt_\N\ar[r,"a"]&
S\times \A_\N.
\end{tikzcd}
\]
By (Loc),
it suffices to show $f_*i_!i^*f^!\simeq 0$,
Since $q$ is strict \'etale,
we have $q^*\simeq q^!$.
Hence it suffices to show $p_*a_*q_*'q'^*a^*p^!\simeq 0$.
Since $\ul{pa}$ is an isomorphism,
there exists a unique cartesian square
\[
\begin{tikzcd}
X\times_{\A_\N}\pt_\N\ar[d,"q'"']\ar[r]&
S'\ar[d,"g"]
\\
S\times \pt_\N\ar[r,"pa"]&
S.
\end{tikzcd}
\]
Observe that $g$ is strict \'etale.

Consider the induced commutative diagram with cartesian squares
\[
\begin{tikzcd}
X\times_{\A_\N}\pt_\N\ar[d,"q'"']\ar[r,"a'"]&
S'\times \A_\N\ar[d,"g'"]\ar[r,"p'"]&
S'\ar[d,"g"]
\\
S\times \pt_\N
\ar[r,"a"]&
S\times \A_\N\ar[r,"p"]&
S.
\end{tikzcd}
\]
We need to show $p_*a_*q_*'q'^*a^*p^!\simeq 0$,
which is equivalent to
$g_*p_*'a_*'a'^*g'^*p^!\simeq 0$.
Since $g$ and $g'$ are strict \'etale,
we have $g^*\simeq g^!$ and $g'^*\simeq g'^!$.
Hence it suffices to show $p_*'a_*'a'^*p'^!\simeq 0$.
This is equivalent to $p_!'a_!'a'^*p'^!\simeq 0$ since $pa$ is proper.
Lemma \ref{Lefschetz.5} and (Loc) finish the proof.
\end{proof}

\begin{lem}
\label{Lefschetz.7}
Let $f\colon X\to S$ be a log smooth morphism in $\lSch/B$ such that $S$ has the trivial log structure,
and let $j\colon X-\partial X\to X$ be the obvious open immersion.
Then the natural transformation
\[
f_*j_!j^!f^!
\xrightarrow{ad'}
f_*f^!
\]
is an isomorphism.
\end{lem}
\begin{proof}
Let $u\colon U\to X$ be an open immersion.
To show that the natural transformation
\[
f_*u_*u^*j_!j^!f^!
\xrightarrow{ad'}
f_*u_*u^*f^!
\]
is an isomorphism,
it suffices to show the claim for $fu$ by ($\eSm$-BC).
Hence we see that the claim is Zariski local on $X$.

Let $v\colon V\to X$ be a dividing cover.
Since ($\divi$-inv) implies $v_*(fv)^!\simeq f^!$,
it suffices to show that the natural transformation
\[
f_*j_!j^!v_*(fv)^!
\xrightarrow{ad'}
f_*v_*(fv)^!
\]
is an isomorphism.
By ($\eSm$-BC) and (Supp),
it suffices to show the claim for $fv$.
Hence we see that the claim is dividing local on $X$.
Together with the Zariski locality on $X$,
we may assume that $X$ has a chart $\N^d$ for some integer $d\geq 0$ such that the induced morphism $X\to S\times \A_{\N^d}$ is strict \'etale.
Lemma \ref{Lefschetz.6} finishes the proof.
\end{proof}

The next result is Lefschetz duality for the case when the target has the trivial log structure:

\begin{thm}
\label{Lefschetz.4}
Let $f\colon X\to S$ be a log smooth morphism in $\lSch/B$ such that $S$ has the trivial log structure,
and let $j\colon X-\partial X \to X$ be the obvious open immersion.
Then the natural transformation
\[
j_!j^!f^!
\xrightarrow{ad'}
f^!
\]
is an isomorphism.
\end{thm}
\begin{proof}
It suffices to show that the induced morphism
\[
\Hom_{\sT^{ex}(X)}
(M_X(V),j_!j^!f^!\cF)
\to
\Hom_{\sT^{ex}(X)}
(M_X(V),f^!\cF)
\]
is an isomorphism for every vertical exact log smooth morphism $v\colon V\to X$ and $\cF\in \sT^{ex}(S)$.
By Zariski descent,
we can work Zariski locally on $V$.
Hence we may assume that the tangent bundle $T_v$ is isomorphic to the rank $d$ trivial bundle.
Proposition \ref{comp.5} and
Theorem \ref{poincare.1} yield natural isomorphisms
\begin{equation}
\label{Lefschetz.4.1}
v^!(-d)[-2d]
\simeq
\Omega_v v^!
\simeq
v^*.
\end{equation}

By adjunction,
it suffices to show that the natural transformation
\[
f_*v_*v^*f^!
\to
f_*v_*v^*j_!j^!f^!
\]
is an isomorphism.
Let $j'\colon V-\partial V\to V$ be the obvious open immersion.
By ($\eSm$-BC) and \eqref{Lefschetz.4.1},
it suffices to show that the natural transformation
\[
f_*v_*v^!f^!
\to
f_*v_* j_!'j'^! v^!f^!
\]
is an isomorphism.
Replace $X$ by $V$ to reduce to Lemma \ref{Lefschetz.7}.
\end{proof}

As a consequence,
we obtain Poincar\'e duality for boundary fs log schemes as follows.

\begin{cor}
\label{Lefschetz.8}
Let $f\colon X\to S$ be a log smooth morphism in $\lSch/B$ such that $S$ has the trivial log structure,
and let $i\colon \partial X \to X$ be the obvious strict closed immersion.
We set $g:=fi$.
Then $g^*$ admits a left adjoint $g_\sharp$,
and there exists a natural isomorphism
\[
g_\sharp\cF
\simeq
g_! (\cF\otimes \Thom(i^*T_f))[-1]
\]
for $\cF\in \sT^{ex}(\partial X)$,
where $T_f$ is the tangent bundle of $f$.
\end{cor}
\begin{proof}
Let $j\colon X-\partial X\to X$ be the obvious open immersion.
Consider the induced commutative square
\[
\begin{tikzcd}
\Omega_f^n j_\sharp j^! f^!\ar[d]\ar[r,"\simeq"]&
j_\sharp \Omega_{fj}^nj^!f^!\ar[r,"\simeq"]\ar[d]&
j_\sharp j^*f^*\ar[d]
\\
\Omega_f^n j_*j^!f^!\ar[r]&
j_* \Omega_{fj}^nj^!f^!\ar[r,"\simeq"]&
j_* j^*f^*
\end{tikzcd}
\]
whose horizontal arrows are isomorphisms by Proposition \ref{exchange.10} and Theorem \ref{poincare.1}.
Since $j_*j^*f^*\simeq f^*$ by ($\ver$-inv),
the cofiber of $j_\sharp j^* f^*\to j_*j^* f^*$ is isomorphic to $i_*i^*f^*$ by (Loc).
Since $j_\sharp j^!f^!\simeq f^!$ by Theorem \ref{Lefschetz.4},
the fiber of $\Omega_f^n j_\sharp j^! f^!\to \Omega_f^n j_* j^* f^!$ is isomorphic to $\Omega_f i_*i^!f^!$.
Hence we have a natural isomorphism
\[
\Omega_f i_*g^![1]
\simeq
i_*g^*.
\]
Apply $i^*$ to this and use Proposition \ref{exchange.10} and (Loc) to have a natural isomorphism
\[
\Thom'(p,a) g^![1]
\simeq g^*,
\]
where $p\colon i^*T_f\to \partial X$ is the projection and $a\colon \partial X \to i^* T_f$ is the zero section.
This shows that $g^*$ admits a left adjoint $g_\sharp$,
and we obtain the desired natural isomorphism by adjunction.
\end{proof}

\section{Cohomology theories of fs log schemes}

Throughout this section,
we fix a log motivic $\infty$-category $\sT$
such that $M_S(X)$ is compact in $\sT^{ex}(X)$ for every vertical exact log smooth morphism $X\to S$,
or that $\sT^{ex}$ satisfies strict \'etale descent.

In \S \ref{orientation},
we provide Thom isomorphisms when an oriented commutative algebra object is given.
In \S \ref{BM},
we define Borel-Moore $\E$-homology for log schemes when $\E$ is an object of $\infSH(B)$.
Based on this,
we also define Borel-Moore motivic homology, G-theory, and Chow homology of separated fs log schemes (over a perfect field).
In \S \ref{logChow},
we define Chow motives over fs log schemes,
and we suggest a strategy to compare the category of Chow motives modulo homological equivalences with the construction in \cite{MR4043075}.
We also show that the category of Chow motives over the standard log point contains $\unit(d)[n]$ as objects for all integers $d$ and $n$.

\subsection{Orientations}
\label{orientation}

A commutative algebra object $\E$ of $\infSH(B)$ is \emph{oriented} if there exists an element $c_\infty \in \E^{2,1}(\P^\infty)$ such that $(c_\infty)|_{\P^1}\simeq \Sigma^{2,1}\unit$.
This is equivalent to the definition of Panin-Pimenov-R\"ondig \cite[Definition 1.2]{MR2475610} after changing the sign.

\begin{exm}
\label{orientation.1}
By \cite[Proposition 11.1]{MR3865569},
$\mathrm{M}\Z$ admits a canonical orientation in $\infSH(B)$.
This implies that $\ML$ admits a canonical orientation in $\infSH(B)$ too for every commutative ring $\Lambda$.
\end{exm}

\begin{prop}
\label{orientation.2}
Let $\E$ be an oriented commutative algebra object of $\infSH(B)$,
and let $\cE$ be a rank $n$ vector bundle over $S\in \lSch/B$.
Then there exists a canonical \emph{Thom isomorphism}
\[
\mathfrak{t}_{\cE}
\colon
\Thom(\cE)
\xrightarrow{\simeq}
\unit(n)[2n]
\]
in $\Mod_{\E}(S)$ satisfying the following properties:
\begin{enumerate}
\item[\textup{(1)}]
If $\cE\to \cF$ is an isomorphism of rank $n$ vector bundles over $S$,
then the induced triangle
\[
\begin{tikzcd}
\Thom(\cE)\ar[rd,"\mathfrak{t}_{\cE}"']\ar[rr]&
&
\Thom(\cF)\ar[ld,"\mathfrak{t}_{\cF}"]
\\
&
\unit(n)[2n]
\end{tikzcd}
\]
commutes.
\item[\textup{(2)}]
If $f\colon X\to S$ is a morphism in $\lSch/B$ and $\cE$ is a rank $n$ vector bundle over $S$,
then the induced square
\[
\begin{tikzcd}
f^*\Thom(\cE)\ar[d,"\simeq"']\ar[r,"\mathfrak{t}_{\cE}"]&
f^*(\unit(n)[2n])\ar[d,"\simeq"]
\\
\Thom(f^*\cE)\ar[r,"\mathfrak{t}_{f^*\cE}"]&
\unit(n)[2n]
\end{tikzcd}
\]
commutes.
\item[\textup{(3)}]
If $0\to \cE'\to \cE\to \cE''\to 0$ is an exact sequence of vector bundles over $S$ and $\cE'$ and $\cE''$ have ranks $m$ and $n$,
then the induced square
\[
\begin{tikzcd}
\Thom(\cE)\ar[r,"\mathfrak{t}_{\cE}"]\ar[d,"\simeq"']&
\unit (m+n)[2m+2n]\ar[d,"\id"]
\\
\Thom(\cE')\otimes \Thom(\cE'')\ar[r,"\mathfrak{t}_{\cE}\otimes \mathfrak{t}_{\cF}"]&
\unit (m+n)[2m+2n]
\end{tikzcd}
\]
commutes,
where the left vertical arrow is obtained by \eqref{comp.4.1}.
\end{enumerate}
\end{prop}
\begin{proof}
If $S$ has the trivial log structure,
then this is a consequence of D\'eglise's results in \cite{MR2466188},
see also \cite[Example 2.4.40]{CD19}.

For general $S$,
let $p\colon S\to \ul{S}$ be the morphism removing the log structure.
Then use the pullback functor $p^*$.
\end{proof}

\subsection{Borel-Moore homology}
\label{BM}

A natural application of the motivic six-functor formalism is defining Borel-Moore homology,
which we discuss in this section.

\begin{df}
\label{BM.1}
Let $\E$ be an object of $\infSH(B)$.
For $X\in \lSch/B$,
\emph{$\E$-cohomology of $X$} is
\[
\E(X)
:=
\hom_{\infSH(X)}(\unit,\E).
\]
For a morphism $f\colon X\to S$ in $\lSch/B$,
\emph{Borel-Moore $\E$-homology of $X$ relative to $S$} is
\[
\E^{\mathrm{BM}}(X/S)
:=
\hom_{\infSH(S)}(f_!\unit,\E).
\]
\end{df}

\begin{df}
Let $X$ be an fs log scheme,
and let $p$ and $q$ be integers.
Recall from  \cite[\S 4.4]{logA1} that we define the \emph{motivic cohomology}, \emph{homotopy $K$-theory}, and \emph{Chow cohomology of $X$} as follows:
\begin{gather*}
R\Gamma_\mathrm{mot}(X,\Z(q))
:=
\hom_{\infSH(X)}(\unit,\MZ(q)),
\\
\KH(X)
:=
\hom_{\infSH(X)}(\unit,\KGL),
\\
\CH^p(X)
:=
\pi_0(\hom_{\infSH(X)}(\unit,\MZ(p)[2p])).
\end{gather*}
\end{df}

\begin{prop}
\label{logChow.13}
Let $X$ be a separated scheme,
and let $f\colon X\to \Spec(\Z)$ be the structure morphism.
Then there exists a natural isomorphism
\[
\rG(X)
\simeq
\hom_{\infSH(\Z)}(f_! \unit,\KGL),
\]
where $\rG(X)$ denotes the $G$-theory spectrum of $X$.
\end{prop}
\begin{proof}
This is due to Jin:
The isomorphisms in \cite[Equation (3.3.3.2)]{1806.03927} can be promoted to isomorphisms of spectra,
which produce the desired isomorphism.
\end{proof}

\begin{df}
\label{logChow.14}
Let $X$ be a separated fs log scheme,
and let $f\colon X\to \Spec(\Z)$ be the structure morphism.
We define the \emph{$G$-theory of $X$} as follows:
\[
\rG(X)
:=
\hom_{\infSH(\Z)}(f_!\unit,\KGL).
\]
By Proposition \ref{logChow.13},
this $\rG(X)$ is an extension of the original $G$-theory spectrum to separated fs log schemes.

Let $X$ be a separated fs log scheme over a perfect field $k$,
and let $f\colon X\to \Spec(k)$ be the structure morphism.
We define the \emph{Borel-Moore motivic homology} and \emph{Chow homology of $X$} as follows:
\begin{gather*}
R\Gamma_\mathrm{mot}^\BM(X,\Z(q))
:=
\hom_{\infSH(k)}(f_!\unit(q),\MZ),
\\
\CH_p(X)
:=
\pi_0(\hom_{\infSH(k)}(f_!\unit(p)[2p],\MZ)),
\end{gather*}
\end{df}

\begin{rmk}
\label{logChow.16}
It is important to distinguish the Chow cohomology and Chow homology even for log smooth fs log schemes.
For example,
we have $\CH^0(\square_k)\simeq \Z$ and $\CH_1(\square_k)\simeq 0$ for every perfect field $k$,
so Poincar\'e duality
\[
\CH^0(\square_k)
\simeq
\CH_1(\square_k)
\]
does not hold.
This is a natural phenomenon since not Poincar\'e duality but Lefschetz duality in algebraic topology is for topological manifolds without boundary.
\end{rmk}

\begin{exm}
Let $f\colon \pt_{\N,k}\to \Spec(k)$, $g\colon \A_{\N,k}\to \Spec(k)$, and $h\colon \G_{m,k}\to \Spec(k)$ be the projections,
where $k$ is a perfect field.
Apply \cite[Corollary 3.6.8]{logshriek} to
\[
\begin{tikzcd}
\pt_{\N,k}\ar[d]\ar[r]&
\A_{\N,k}\ar[d]
\\
\Spec(k)\ar[r]&
\A_k^1
\end{tikzcd}
\]
to have a natural isomorphism $f_*f^*\simeq g_*g^*$.
We also have a natural isomorphism $g_*g^*\simeq h_*h^*$ by ($\ver$-inv).
It follows that we have $f_*\unit \simeq \unit \oplus \unit(-1)[-1]$.
Hence we have
\[
\CH_i(\pt_{\N,k})
\simeq
\left\{
\begin{array}{ll}
\Z\oplus k^* & \text{if }i=0,
\\
0 & \text{otherwise}.
\end{array}
\right.
\]
By Corollary \ref{Lefschetz.8},
we also have
\begin{equation}
\label{logChow.16.1}
f_\sharp \unit \simeq \unit \oplus \unit(1)[1].
\end{equation}
\end{exm}

\subsection{Chow motives}
\label{logChow}
In this section,
we define Chow motives over fs log schemes and discuss basic properties.

\begin{df}
\label{logChow.9}
For $S\in \lSch/B$,
the \emph{category of Chow motives $\Chow(S)$ over $S$} is the idempotent completion of the full subcategory of the homotopy category of $\ModMZ(S)$ spanned by $M_S(X)(r)[2r]$ for every projective exact vertical log smooth morphism $X\to S$ and integer $r$.
\end{df}

\begin{rmk}
\label{logChow.11}
Let $k$ be a field.
If $X\to \Spec(k)$ is a projective exact vertical log smooth morphism of fs log schemes,
then it is a projective smooth morphism of schemes.
Hence $\Chow(k)$ is equivalent to Grothendieck's category of Chow motives by Proposition \ref{logChow.10} below.
\end{rmk}

\begin{prop}
\label{logChow.10}
Let $f\colon X\to S$ and $g\colon Y\to S$ be projective exact vertical log smooth morphisms in $\lSch/B$.
Assume that $g$ has relative dimension $n$.
Then there exists a natural isomorphism
\[
\Hom_{\Chow(S)}(M_S(X)(r)[2r],M_S(Y)(s)[2s])
\simeq
\CH^{n+s-r}(X\times_S Y)
\]
for integers $r$ and $s$.
\end{prop}
\begin{proof}
By Theorem \ref{poincare.1} and Proposition \ref{orientation.2},
we have an isomorphism $M_S(Y)\simeq g_*\unit (n)[2n]$ in $\Mod_{\MZ}(S)$.
Hence we have an isomorphism
\begin{align*}
& \Hom_{\ModMZ(S)}(M_S(X)(r)[2r],M_S(Y)(s)[2s])
\\
\simeq &
\Hom_{\ModMZ(X)}(M_X(X\times_S Y)(r)[2r],\unit(n+s)[2(n+s)])
\\
\simeq &
\Hom_{\ModMZ(X\times_S Y)}(\unit,\unit(n+s-r)[2(n+s-r)]).
\end{align*}
From this, we obtain the desired isomorphism.
\end{proof}

In \cite[\S 3.1.2]{MR4043075},
Ito, Kato, Nakayama, and Usui define the category of log motives $\mathrm{LM}(S)$ for any fs log scheme $S$.
See \cite[Appendix A]{MR4938005} for variants of $\mathrm{LM}(S)$.
For simplicity of the exposition,
assume that $S=\pt_{\N,k}$ with an algebraically closed field $k$.
Then $\mathrm{LM}(S)$ can be described as follows.
The objects of $\mathrm{LM}(S)$ are the symbols $h(X)(r)$ for all projective vertical log smooth fs log scheme $X$ over $S$ and integer $r$.
In this case,
$X$ is exact over $S$.
For projective vertical log smooth fs log schemes $X$ and $Y$ over $S$ and integers $r$ and $s$,
$\Hom_{\mathrm{LM}(S)}(h(X)(r),h(Y)(s))$ is the image of the Chern class map
\begin{equation}
\label{logChow.1.1}
\mathrm{gr}^i K_{\mathrm{lim}}(X\times_S Y)_\Q
\to
H_{l\et}^i (X\times_S Y,\Q_\ell(i)),
\end{equation}
with $i:=\dim(X)+s-r$ and prime $\ell$.
Here, $\mathrm{gr}^i$ is the $i$th graded piece for the $\gamma$-filtration on $K$,
and
\[
K_{\mathrm{lim}}(V):=\colim_{V'\in V_{\divi}}K(V)
\]
for every noetherian fs log scheme $V$.
Observe that we have a natural isomorphism
\[
H_{l\et}^i (X\times_S Y,\Q_\ell(i))
\simeq
H_{k\et}^i (X\times_S Y,\Q_\ell(i))
\]
between log \'etale cohomology and Kummer \'etale cohomology by \cite[Proposition 5.4(2)]{MR3658728}.

\begin{quest}
\label{logChow.2}
Let $k$ be a perfect field,
let $S\in \lSch/k$,
and let $\ell$ be a prime invertible in $k$.
Using the Kummer \'etale realization functor
\[
\Mod_{\MZ}(S)\to \rD_\mathrm{ket}(S,\Lambda)
\]
to the derived $\infty$-category of Kummer \'etale sheaves on the small \'etale site $S_\mathrm{ket}$ obtained by \cite[Construction 9.9]{logetale},
we obtain the natural homomorphism
\begin{equation}
\label{logChow.2.1}
\Hom_{\Chow(S)}(M_S(X)(r)[2r],M_S(Y)(s)[2s])
\to
H_{k\et}^i(X\times_S Y,\Q_\ell(i)),
\end{equation}
where $i:=n+s-r$ and $n$ is the relative dimension of $Y\to S$.
Let $\Chow_\mathrm{hom}(S)$ be the category of $M_S(X)(r)[2r]$ for all projective exact vertical log smooth morphism $X\to S$ and integer $r$ such that the morphism set is replaced with the image of \eqref{logChow.2.1}.

Let $k$ be an algebraically closed field.
Are the images of \eqref{logChow.1.1} and \eqref{logChow.2.1} are equal if $S=\pt_{\N,k}$?
An affirmative answer implies that there is an equivalence of categories
\[
\mathrm{LM}(\pt_{\N,k})
\simeq
\Chow_\mathrm{hom}(\pt_{\N,k}).
\]
\end{quest}

\begin{exm}
\label{logChow.12}
Let us explain how to construct the toroidal model of an elliptic curve.
Let $\ul{E}$ be the gluing of two copies $\ul{L_1}$ and $\ul{L_2}$ of $\P_k^1$ along the two points $0$ and $\infty$,
where $k$ is a field.
Then $\ul{U_1}:=\ul{E}-\{0\}$ and $\ul{U_2}:=\ul{E}-\infty$ are isomorphic to $\Spec(k[x,y]/(xy))$.
We impose the log structures on $\ul{U_1}$ and $\ul{U_2}$ given by $\Spec(\N x\oplus \N y\to k[x,y]/(xy))$,
and the log structures are compatible along $\ul{U_1}\cap \ul{U_2}$.
Observe that $\ul{U_1}\cap \ul{U_2}$ is isomorphic to two copies $\ul{V_1}$ and $\ul{V_2}$ of $\G_m$ with $\ul{V_1}\subset \ul{L_1}$ and $\ul{V_2}\subset \ul{L_2}$.
Let $X$, $L_1$, $L_2$, $U_1$, $U_2$, $V_1$, and $V_2$ be the resulting fs log schemes with the corresponding underlying schemes.
We can regard $E$ as an object of $\lSm/\pt_{\N,k}$.
This $E$ is the toroidal model of an elliptic curve.

We have isomorphisms $\ul{V_1}\simeq \Spec(k[x,y^\pm]/(xy))$ and $\ul{V_2}\simeq \Spec(k[x^\pm,y]/(xy))$,
which yield isomorphisms $V_1\xleftarrow{v_1} \pt_{\N,k}\times  \G_m \xrightarrow{v_2} V_2$,
where $v_1$ (resp.\ $v_2$) are given by $x\mapsto t$ (resp.\ $y\mapsto t$) with the coordinate $t$ for $\G_m$.
Let $\tau\colon U_1\to U_1$ be the automorphism switching $x$ and $y$,
which satisfies $\tau^2=\id$.
Then we have a commutative diagram
\[
\begin{tikzcd}
M_{\pt_{\N,k}}(V_1)\ar[d]\ar[r,leftarrow,"\simeq"]&
M_{\pt_{\N,k}}(\pt_{\N,k}\times \G_m)\ar[r,"\simeq"]&
M_{\pt_{\N,k}}(V_2)\ar[d]
\\
M(U_1)\ar[rr,"\tau"]&
&
M(U_1)
\end{tikzcd}
\]
whose vertical morphisms are isomorphisms by \cite[Proposition 2.4.2]{logshriek}.
We have a similar result for $U_2$.
Hence the cofiber sequence
\[
M_{\pt_{\N,k}}(U_1\cap U_2)
\to
M_{\pt_{\N,k}}(U_1)\oplus M_{\pt_{\N}}(U_2)
\to
M_{\pt_{\N,k}}(E)
\]
becomes a cofiber sequence
\begin{align*}
& M_{\pt_{\N,k}}(\pt_{N,k}\times \G_m)\oplus M_{\pt_{\N,k}}(\pt_{N,k}\times \G_m)
\\
\xrightarrow{A} &
M_{\pt_{\N,k}}(\pt_{N,k}\times \G_m)\oplus M_{\pt_{\N,k}}(\pt_{N,k}\times \G_m)
\to
M_{\pt_{\N,k}}(E),
\end{align*}
where $A$ is given by a matrix
\[
\left(
\begin{array}{cc}
\id & \sigma\\
\sigma & \id
\end{array}
\right)
\]
for some automorphism $\sigma$ such that $\sigma^2=\id$.
Since the second row is the multiple of the first row by $\sigma$,
we obtain isomorphisms
\[
M_{\pt_{\N,k}}(E)
\simeq
M_{\pt_{\N,k}}(\pt_{N,k}\times \G_m)
\oplus
M_{\pt_{\N,k}}(\pt_{N,k}\times \G_m)[1]
\simeq
\unit \oplus \unit[1] \oplus \unit(1)[1] \oplus \unit(1)[2]
\]
in $\Mod_{\MZ}(\pt_{\N,k})$.
It follows that the circles $\unit[1]$ and $\unit(1)[1]$ are objects of $\Chow(\pt_{\N,k})$.
Furthermore,
we have the following computation using \eqref{logChow.16.1}:
\begin{align*}
\CH^1(E)
= &
\pi_0\Hom_{\infSH(\pt_{\N,k})}(\unit \oplus \unit[1]\oplus \unit(1)[1]\oplus \unit(1)[2],\unit(1)[2])
\\
\simeq &
\pi_0\Hom_{\infSH(k)}((\unit \oplus \unit[1]\oplus \unit(1)[1]\oplus \unit(1)[2])\otimes (\unit \oplus \unit(1)[1]),\unit(1)[2])
\\
\simeq &
k^*\oplus \Z\oplus \Z.
\end{align*}
This is not isomorphic to the prelog Chow group $\CH_\mathrm{prelog}^1(E)\simeq \Z$ in \cite[Definition 2.4] {MR4609665}.
\end{exm}

The following is a notable difference with Grothendieck's category of Chow motives.

\begin{prop}
\label{logChow.15}
For all integers $d$ and $n$,
we have
\[
\unit(d)[n]
\in
\Chow(\pt_{\N}).
\]
\end{prop}
\begin{proof}
As in Example \ref{logChow.12},
we have $\unit[1],\unit(1)[1]\in \Chow(\pt_\N)$.
We also have $\unit(-1)[-2]\in \Chow(\pt_\N)$.
We deduce the claim by tensoring these.
\end{proof}

\bibliography{bib}

\begin{thebibliography}{10}

\bibitem{1810.03746}
{\sc L.~J. Barrott}, {\em Bivariant classes and logarithmic {C}how theory}.
\newblock Preprint, \href{https://arxiv.org/pdf/1810.03746v3.pdf}{arXiv:1810.03746v3}, 2020.

\bibitem{logDM}
{\sc F.~Binda, D.~Park, and P.~A. {\O}stv{\ae}r}, {\em Triangulated categories of logarithmic motives over a field}, Ast{\'e}risque, 433 (2022).

\bibitem{logSH}
\leavevmode\vrule height 2pt depth -1.6pt width 23pt, {\em Logarithmic motivic homotopy theory}.
\newblock Preprint, \href{https://arxiv.org/pdf/2303.02729v3.pdf}{arXiv:2303.02729v3}, to appear in Mem. Am. Math. Soc., 2025.

\bibitem{MR4609665}
{\sc C.~B\"{o}hning, H.-C.~G. von Bothmer, and M.~van Garrel}, {\em Prelog {C}how rings and degenerations}, Rend. Circ. Mat. Palermo (2), 72 (2023), pp.~2861--2894.

\bibitem{CD19}
{\sc D.-C. {Cisinski} and F.~{D\'eglise}}, {\em {Triangulated categories of mixed motives}}, Cham: Springer, 2019.

\bibitem{CLStoric}
{\sc D.~Cox, J.~Little, and H.~Schenck}, {\em Toric Varieties}, Graduate studies in mathematics, American Mathematical Soc., 2011.

\bibitem{MR2466188}
{\sc F.~D\'{e}glise}, {\em Around the {G}ysin triangle. {II}}, Doc. Math., 13 (2008), pp.~613--675.

\bibitem{2108.02845}
{\sc O.~Gregory and A.~Langer}, {\em A log-motivic cohomology for semistable varieties and its $p$-adic deformation theory}.
\newblock Preprint, \href{https://arxiv.org/pdf/2108.02845v4.pdf}{arXiv:2108.02845v5}, 2025.

\bibitem{EGAIVIV}
{\sc A.~Grothendieck}, {\em \'el\'ements de g\'eom\'etrie alg\'ebrique. {IV}. \'etude locale des sch\'emas et des morphismes de sch\'emas {IV}}, Inst. Hautes \'Etudes Sci. Publ. Math.,  (1967), p.~361.

\bibitem{MR4043075}
{\sc T.~Ito, K.~Kato, C.~Nakayama, and S.~Usui}, {\em On log motives}, Tunis. J. Math., 2 (2020), pp.~733--789.

\bibitem{1806.03927}
{\sc F.~Jin}, {\em Algebraic {G}-theory in motivic homotopy categories}.
\newblock Preprint, \href{https://arxiv.org/pdf/1806.03927v3.pdf}{arXiv:1806.03927v3}, 2019.

\bibitem{zbMATH01367783}
{\sc K.~Kato and C.~Nakayama}, {\em log {Betti} cohomology, log {\'e}tale cohomology, and log de {Rham} cohomology of log schemes over {{\(\mathbb{C}\)}}}, Kodai Math. J., 22 (1999), pp.~161--186.

\bibitem{MR4938005}
{\sc K.~Kato, C.~Nakayama, and S.~Usui}, {\em Mixed objects are embedded into log pure objects}, Osaka J. Math., 62 (2025), pp.~445--465.

\bibitem{HA}
{\sc J.~{Lurie}}, {\em Higher algebra}.
\newblock \url{https://www.math.ias.edu/~lurie/papers/HA.pdf}, 2017.

\bibitem{MVW}
{\sc C.~Mazza, V.~Voevodsky, and C.~Weibel}, {\em Lecture notes on motivic cohomology}, vol.~2 of Clay Mathematics Monographs, American Mathematical Society, Providence, RI; Clay Mathematics Institute, Cambridge, MA, 2006.

\bibitem{MR4549706}
{\sc S.~Molcho, R.~Pandharipande, and J.~Schmitt}, {\em The {H}odge bundle, the universal 0-section, and the log {C}how ring of the moduli space of curves}, Compos. Math., 159 (2023), pp.~306--354.

\bibitem{MV}
{\sc F.~Morel and V.~Voevodsky}, {\em {${\bf A}^1$}-homotopy theory of schemes}, Inst. Hautes \'{E}tudes Sci. Publ. Math.,  (1999), pp.~45--143 (2001).

\bibitem{MR3658728}
{\sc C.~Nakayama}, {\em Logarithmic \'{e}tale cohomology, {II}}, Adv. Math., 314 (2017), pp.~663--725.

\bibitem{Ogu}
{\sc A.~Ogus}, {\em Lectures on Logarithmic Algebraic Geometry}, Cambridge Studies in Advanced Mathematics, Cambridge University Press, 2018.

\bibitem{MR2475610}
{\sc I.~Panin, K.~Pimenov, and O.~R\"{o}ndigs}, {\em A universality theorem for {V}oevodsky's algebraic cobordism spectrum}, Homology Homotopy Appl., 10 (2008), pp.~211--226.

\bibitem{ParThesis}
{\sc D.~Park}, {\em Triangulated categories of motives over fs log schemes}, PhD thesis, University of California, Berkeely, 2016.

\bibitem{divspc}
\leavevmode\vrule height 2pt depth -1.6pt width 23pt, {\em Inverting log blowups in log geometry}, Tunis. J. Math., 6 (2024), pp.~405--453.

\bibitem{logshriek}
\leavevmode\vrule height 2pt depth -1.6pt width 23pt, {\em Log motivic exceptional direct image functors}.
\newblock Preprint, \href{https://arxiv.org/pdf/2403.06692v3.pdf}{arXiv:2403.06692v3}, 2026.

\bibitem{logGysin}
\leavevmode\vrule height 2pt depth -1.6pt width 23pt, {\em Log motivic {G}ysin isomorphisms}.
\newblock Preprint, \href{https://arxiv.org/pdf/2303.12498v2.pdf}{arXiv:2303.12498v2}, 2026.

\bibitem{logA1}
\leavevmode\vrule height 2pt depth -1.6pt width 23pt, {\em {$\mathbb{A}^1$}-homotopy theory of log schemes}, Eur. J. Math., 12 (2026).
\newblock Paper No. 42.

\bibitem{logetale}
\leavevmode\vrule height 2pt depth -1.6pt width 23pt, {\em Six-functor formalism for {K}ummer étale cohomology of log schemes}.
\newblock Preprint, \href{http://arxiv.org/pdf/2608.29386v1.pdf}{arXiv:2608.29386v1}, 2026.

\bibitem{2209.03720}
{\sc G.~Shuklin}, {\em A {V}oevodsky motive associated to a log scheme}.
\newblock Preprint, \href{https://arxiv.org/pdf/2209.03720v2.pdf}{arXiv:2209.03720v2}, 2024.

\bibitem{MR3865569}
{\sc M.~Spitzweck}, {\em A commutative {$\mathbb{P}^1$}-spectrum representing motivic cohomology over {D}edekind domains}, M\'{e}m. Soc. Math. Fr. (N.S.),  (2018), p.~110.

\bibitem{MR1705837}
{\sc T.~Tsuji}, {\em {$p$}-adic \'{e}tale cohomology and crystalline cohomology in the semi-stable reduction case}, Invent. Math., 137 (1999), pp.~233--411.

\bibitem{zbMATH07027475}
\leavevmode\vrule height 2pt depth -1.6pt width 23pt, {\em Saturated morphisms of logarithmic schemes}, Tunis. J. Math., 1 (2019), pp.~185--220.

\end{thebibliography}
\bibliographystyle{siam}

\end{document}